\documentclass{amsart}
\usepackage{amsthm,amsmath,amssymb}
\usepackage{mathrsfs}

\usepackage{color}
\usepackage{tikz}
\usepackage{bm}
\usepackage{a4wide}
\usepackage{pstricks}

\newtheorem{theorem}{Theorem}[section]
\newtheorem{lemma}[theorem]{Lemma}
 
\theoremstyle{definition}
\newtheorem{definition}[theorem]{Definition}
\theoremstyle{remark}

\newtheorem{example}[theorem]{Example}

\date{}

\newcommand{\leqn}{\preccurlyeq}
\newcommand{\geqn}{\succcurlyeq}
\newcommand{\N}{\mathbb{N}}
\newcommand{\restrict}[1]{\!\upharpoonright_{#1}}

\newcommand{\variety}[1]{\mathfrak{#1}}
\newcommand{\varietyv}[1]{\boldsymbol{\mathfrak{#1}}}
\newcommand{\Type}[1]{\mathrm{Type}_{\, #1}}
\newcommand{\DwSet}[1]{{\downarrow}#1}
\newcommand{\UpSet}[1]{{\uparrow}#1}
\newcommand{\Cn}{\mathcal{C}_n}
\newcommand{\Con}{\mathcal{C}}

\newcommand{\PDL}{\ensuremath{ \varietyv{B}_{\omega}}}
\newcommand{\PDLZero}{\ensuremath{ \varietyv{B}_{0}}}
\newcommand{\PDLOne}{\ensuremath{ \varietyv{B}_{1}}}
\newcommand{\PDLn}{\ensuremath{{\varietyv{B}_{n}}}}
\newcommand{\PDLnf}{\ensuremath{{\varietyv{B}_{n}^f}}}

\newcommand{\POS}{\ensuremath{\boldsymbol{\mathscr{P}}^f}}
\newcommand{\POn}{\ensuremath{\boldsymbol{\mathscr{P}}_{n}^f}}
\newcommand{\pow}[1]{\mathcal{P}(#1)}

\newcommand{\alg}[1]{\mathbf{#1}}
\newcommand{\Bn}{\alg{B}_{n}}
\newcommand{\OBn}{\overline{\alg{B}}_{n}}

\newcommand{\pos}[1]{\mathtt{#1}}
\newcommand{\Pm}{\pos{P}(m)}
\newcommand{\POne}{\pos{P}(1)}
\newcommand{\Un}[1]{{\pos{U}}_{#1}}
\newcommand{\X}{{\pos{X}}}
\newcommand{\Y}{{\pos{Y}}}
\newcommand{\Z}{{\pos{Z}}}

\newcommand{\D}{\mathsf{J}}
\newcommand{\E}{\mathsf{D}}

\renewcommand{\u}{\upsilon}

\newcommand{\Cstar}{\ensuremath{(*)}}
\newcommand{\Cnstar}{\ensuremath{(*_n)}}

\title{Unification on subvarieties of pseudocomplemented distributive lattices}

\author{Leonardo Cabrer}

\address{University of Oxford\\Mathematical Institute\\
24-29 St. Giles\\
Oxford (OX1 3LB)}

\begin{document}

\begin{abstract}
In this paper subvarieties of pseudocomplemented distributive lattices are classified by their unification type. We determine the unification type of every particular unification problem in each subvariety of pseudocomplemented distributive lattices.
\end{abstract}
\maketitle

\section{Introduction}

Syntactic unification theory is concerned with the problem of finding a substitution that equalises a finite set of pairs of terms simultaneously. 
More precisely, given a set of function symbols $\mathcal{L}$ and a finite set of pairs of $\mathcal{L}$-terms ${U=\{(t_1,s_1),\ldots,(t_m,s_m)\}}$, called a {\em unification problem}, a {\it unifier} for $U$ is a substitution $\sigma$ defined on the set of variables of the terms in $U$ such that $\sigma(t_i)=\sigma(s_i)$ for each $i\in\{1,\ldots,m\}$.
In many applications the operations in $\mathcal{L}$ are assumed to satisfy certain conditions that can be expressed  by equations, such as associativity, commutativity, idempotency. 
Then  syntactic unification evolves into {\it equational unification}. 
Given an equational theory $E$ in the language $\mathcal L$, a unifier for $U$ is now  asked to send the terms in each pair $(t_i,s_i)\in U$ to terms  $\sigma(t_i)$ and $\sigma(s_i)$ that are equivalent for $E$ (in symbols, $\sigma(t_i)\approx_{E}\sigma(s_i)$).

Once a particular unification problem is known to admit $E$-unifiers, the next task is to find a complete description  of its unifiers. 
For that we first observe that if $\sigma$ is an $E$-unifier for $U$, then $\gamma\circ\sigma$ is also an  $E$-unifier for $U$, whenever $\gamma$ is a substitution such that $\gamma\circ\sigma$ is well defined.
In this case we say that $\sigma$ is {\it more general} than $\gamma\circ\sigma$.
Therefore, a useful way to determine all the unifiers of a particular problem is to calculate a family of unifiers such that 
any other unifier of the problem is less general than one of the unifiers of the family. 
This set is called a {\it complete set of unifiers}. 
It is desirable to obtain a complete set that is not `redundant' (in the sense that  the elements of the set are incomparable). Any such set is called a {\it minimal complete set of unifiers}. 
The {\it unification type} of a unification problem is defined depending on the existence and the cardinality of a minimal complete set of unifiers (see Section~\ref{Sec:Preliminaries}). 
(We refer the reader to the surveys \cite{BSi1994,BSn2001,JK1991} for detailed definitions, historical references and applications of unification theory.)

Unification problems related to extensions of Intuitionistic Propositional Logic (Intermediate Logics) and their fragments, that is, the equational theory of subvarieties of Heyting algebras and their reducts, have been studied by several authors. The equational theory of Heyting algebras has been proved to be finitary (that is, each unification problem admits a finite minimal complete set of unifiers) by Ghilardi in \cite{Gh1999}. The unification type of various subvarieties of Heyting algebras has been determined in \cite{Dz2006,Gh1999,Gh2004,Wr1995}. Unification in different fragments of intuitionistic logic that include the implication were investigated in \cite{CM2010,IR201X,MW1988,Pr1973}. The variety of bounded distributive lattices was proved to have nullary  type (that is, there exists a unification problem that does not admit a minimal complete set of unifiers) in~\cite{Gh1997}, and the type of each unification problem was calculated in \cite{BC201X}.

We devote this paper to the study of unification in the implication-free fragment of intuitionistic logic, that is, the equational theory of pseudocomplemented distributive lattices (\emph{$p$-lattices} for short) and its extensions. 
It was first observed by Ghilardi in \cite{Gh1997} that the equational theory of $p$-lattices has nullary type. 
In this paper we take that result two steps forward. First we prove that Boolean algebras form the only non-trivial subvariety of $p$-lattices that has type one, while the others have nullary type. 
Secondly, we determine the type of each unification problem in every extension of  the equational theory of $p$-lattices.

The main tools used in this paper are: 
the algebraic approach to $E$-unification developed in \cite{Gh1997}; 
the categorical duality for bounded distributive lattices presented in \cite{Pr1970}, and its restriction to $p$-lattices developed in \cite{Pr1975};
 the characterisation of subvarieties of $p$-lattices given in \cite{Le1970}; and the description of finite projective $p$-lattices in these subvarieties given in \cite{Ur1981}.

The paper is structured as follows. We first collect in Section~\ref{Sec:Preliminaries} some preliminary material on subvarieties of $p$-lattices, finite duality for $p$-lattices and algebraic unification theory.
 Section~\ref{Sec:SpecialProd} is devoted to the study of the properties of duals of projective $p$-lattices needed throughout  the rest of this paper. 
Then, in Section~\ref{Sec:UnifType} we determine the unification type of each subvariety of $p$-lattices. 
Finally in Sections~\ref{Sec:UnifPDL1}, \ref{Sec:UnifPDL} and~\ref{Sec:UnifPDLn} we present the algorithms to calculate the unification type of each problem for each subvariety of $p$-lattices.
The statements and proofs of results in  Sections~\ref{Sec:UnifPDL} and~\ref{Sec:UnifPDLn} require specific definitions and preliminaries. We delay the introduction of these definitions to Section~\ref{Sec:Unifcore}, since they are not needed in the previous sections of the paper.

\section{Preliminaries}\label{Sec:Preliminaries}

\subsection*{Unification type}\ 

Let $\mathcal{S}=(S,\leqn)$ be a preordered class, that is, $S$ is a class and $\leqn$ is a reflexive and transitive binary relation on $S$. 
Then $\mathcal{S}$ has a natural category  structure whose objects are the elements of $S$ and whose morphisms are elements of $\leqn$. If $(x,y)\in\leqn$, then $x$ and $y$ are the domain and codomain of $(x,y)$, respectively. 
The relation $\approx\,=\,\leqn\cap\geqn$ is an equivalence relation on $S$. In this paper we deal only with $\mathcal{S}$ such that 
$(S/_{\approx},{\leqn}/_{\approx})$ is isomorphic to a partially ordered set.
A \emph{complete set} for $\mathcal{S}$
is a subset $M$ of  $S$ such that for every $x \in S$ there exists $y \in M$ with $x \leqn y$.  The set $M$ is said to be \emph{minimal complete} for $\mathcal S$ if it is complete and $x \leqn y$ implies $x=y$ for all $x,y \in M$. If $\mathcal{S}$ has a minimal complete set $M$, 
then every minimal complete set of $\mathcal{S}$ has the same cardinality as $M$.
The \emph{type} of the preorder $\mathcal{S}$ is defined as follows:
\[
\Type{}(\mathcal{S})=
\begin{cases}
0 &  \mbox{if }\mathcal{S}\mbox{ has no minimal complete set};\\
\infty & \mbox{if }\mathcal{S}\mbox{ has a minimal complete set of infinite cardinality;}\\
n &\mbox{if }\mathcal{S}\mbox{ has a finite minimal complete set of cardinality }n\mbox{.}
\end{cases}
\]
If two preordered classes are equivalent as categories, then they have the same type. 

We collect here some sufficient conditions on a preordered class to have type $0$.
\begin{theorem}{\rm \cite{Baader}}\label{Theo:Baader}
Let $\mathcal{S}=(S,\leqn)$ be a preordered class.
Then each of the following conditions implies that $\Type{}(\mathcal{S})=0$.
\begin{itemize}
\item[(i)] There is an increasing sequence  $s_1 \leqn s_2 \leqn s_3 \ldots$ in $\mathcal{S}$ without upper bounds in $\mathcal{S}$ having the property: for all $s\in S$ and $n\in\N$, if $s_n\leqn s$, there exists $t \in S$ such that $s \leqn t$ and $s_{n+1}\leqn t$.
\item[(ii)] $\mathcal{S}$ is \emph{directed }{\rm (}for each $x,y\in S$ there exists  $z\in S$ such that  $x,y\leqn z${\rm )} and there is an increasing sequence  $s_1 \leqn s_2 \leqn s_3 \ldots$ in $\mathcal{S}$ without upper bounds in $\mathcal{S}$.
\end{itemize}
\end{theorem}

The algebraic unification theory developed in \cite{Gh1997}
translates the traditional $E$-unification problem into algebraic terms as we describe in what follows.
Let $\variety{V}$ be a variety of algebras. 
An algebra $\alg{A}$ in $\variety{V}$ is said to be \emph{finitely presented} if there exist an $n$-generated free algebra $\alg{Free}_{\variety{V}}(n)$ and a finitely generated congruence $\theta$ of $\alg{Free}_{\variety{V}}(n)$ such that $\alg{A}$ is isomorphic to $\alg{Free}_{\variety{V}}(n)/\theta$.
Recall that a finitely generated algebra $\alg{P}$ is \emph{(regular) projective} in $\variety{V}$ if and only $\alg{P}$ is a retract of a finitely generated free algebra in $\variety{V}$. 
A \emph{unification problem for} $\variety{V}$ is a finitely presented algebra $\alg{A} \in \variety{V}$. An \emph{{\rm (}algebraic{\rm )} unifier in $\variety{V}$} for a finitely presented algebra $\alg{A}\in\variety{V}$ is a homomorphism $u \colon \alg{A} \to \alg{P}$, where $\alg{P}$ is a finitely generated projective algebra in $\variety{V}$. A unification problem $\alg{A}$ is called \emph{solvable in $\variety{V}$} if $\alg{A}$ has a unifier in $\variety{V}$.

Let $\alg{A} \in \variety{V}$ be finitely presented
and for $i=1,2$ let $u_i \colon \alg{A} \to \alg{P}_i$ be a unifier for $\alg{A}$.
Then $u_1$ is \emph{more general} than $u_2$; in symbols, $u_2 \leqn_{\variety{V}} u_1$,
if there exists a homomorphism $f \colon \alg{P}_1 \to \alg{P}_2$
such that $f \circ u_1=u_2$. For $\alg{A}$ solvable in $\variety{V}$, let $\Un{\variety{V}}(\alg{A})$ be the preordered class of unifiers for $\alg{A}$ whose preorder is the relation $\leqn_{\variety{V}}$.  
We shall omit the subscript and write $\leqn$ instead of $\leqn_{\variety{V}}$ when the variety is clear from the context.
We define the \emph{type of $\alg{A}$ in $\variety{V}$} as the type of the preordered class $\Un{\variety{V}}(\alg{A})$, in symbols
$\Type{\variety{V}}(\alg{A})=\Type{}(\Un{\variety{V}}(\alg{A}))$.  

Let $T(\variety{V})=\{ \Type{\variety{V}}(\alg{A}) \mid \alg{A} \text{ solvable in }\variety{V} \}$ be the set of types of solvable problems in $\variety{V}$. The type of the variety $\variety{V}$  is defined depending on  $T(\variety{V})$ as follows:
\[
\Type{}(\variety{V})=\begin{cases}
0 &\mbox{if }0\in T(\variety{V})\mbox{;}\\
\infty& \mbox{if }\infty\in T(\variety{V})\mbox{ and }0\notin T(\variety{V})\mbox{;}\\
\omega& \mbox{if }0,\infty\notin T(\variety{V})\mbox{ and }\\
&\ \forall n\in\N,\exists m\mbox{ such that }m\in T(\variety{V})\mbox{ and }n\leq m\mbox{;}\\
n &\mbox{if }n\in  T(\variety{V}) \mbox{ and }T(\variety{V})\subseteq\{1,\ldots,n\}\mbox{.}
\end{cases} 
\]
 Equivalently, $\Type{}(\variety{V})$ is the supremum of $T(\variety{V})$ in the total order \[1<2<\cdots<\omega<\infty<0.\]
\subsection*{$p$-lattices}\ 

An algebra $\alg{A}=(A,\wedge,\vee,\neg,0,1)$ is said to be a 
\emph{pseudocomplemented distributive lattice} ($p$-lattice) if $(A,\wedge,\vee,0,1)$ is a bounded distributive lattice
 and $\neg a$ is the maximum element of the set $\{b\in A\mid b\wedge a=0\}$
for each $a\in A$. 
Each finite distributive $(L,\wedge,\vee,0,1)$ admits a  unique $\neg$ operation such that $(L,\wedge,\vee,\neg,0,1)$ is a $p$-lattice.

The class of $p$-lattices form a variety, that is, it is closed under products, subalgebras, and homomorphic images (equivalently, it is determined by a set of equations).
In what follows, $\PDL$ denotes both the variety of $p$-lattices and the category of $p$-lattices as objects and homomorphisms as arrows. The variety $\PDL$ is locally finite, that is, every finitely generated algebra is finite. 
We let 
$\PDL^f$ denote the subcategory of finite $p$-lattices.

For each $n\in\{0,1,2,\ldots\}$, let $\Bn=(B_n,\wedge,\vee,^{*},0,1)$ denote the finite Boolean algebra with $n$ atoms and let $\OBn$ be the algebra obtained by adding a new top~$1'$ to the underlying lattice of $\Bn$ and endowing it with the unique operation that upgrades it to a $p$-lattice. More specifically the $\neg$ operation in $\OBn$ is defined as follows: 
$\neg 0=1'$; $\neg 1'=0$; and $\neg a=a^*$ otherwise. Let $\PDLn$ denote the subvariety of $\PDL$ generated by $\OBn$ and
the full subcategory of $\PDL$ formed by its algebras.
In \cite{Le1970},  it is proved 
that every non-trivial proper subvariety of $\PDL$ coincides with some $\PDLn$. 
Observe that $\PDLZero$ and $\PDLOne$ are the varieties of Boolean algebras and  Stone algebras, respectively.
As for the case of $p$-lattices we let $\PDLnf$  denote the full subcategory of $\PDLn$ whose objects have finite universes.

Throughout the paper, we let the symbol $\N$ denote the set of natural numbers $\{1,2,3,\ldots\}$.


\subsection*{Duality for finite $p$-lattices}\ 

In \cite{Pr1975}, a topological duality for $p$-lattices is developed. In this paper we only need its restriction to finite objects, where the topology does not play any role.

Let $\X=(X,\leq)$ be a finite partially ordered set (poset, for short). Given  $Y\subseteq X$ a non-empty subset of $X$,
  let $\Y=(Y,\leq_{Y})$ denote the subposet of $\X$ whose universe is $Y$, that is, the poset such that its order relation is $\leq_{Y}\,=\,\leq\cap Y^2$. Let $\UpSet{Y}=\{x\in X\mid (\exists y\in Y)\ y\leq x\}$ and  $\DwSet{Y}=\{x\in X\mid (\exists y\in Y)\ x\leq y\}$ denote the up-set and down-set generated by $Y$, respectively. 
If $Y=\{y\}$ for some $y\in X$, we simply write $\UpSet{y}$ and $\DwSet{y}$. 
 Let $\min(\X)$ and $\max(\X)$ denote the set of minimal  and maximal elements of $\X$, respectively. Given $x\in X$, the set of minimal elements of $\X$ below $x$ will be denoted by $\min_{\X}(x)=\min(\X)\cap \DwSet{x}$. 


Let $\POS$ be the category whose objects are finite posets and whose arrows are {\it $p$-morphisms}, that is, monotone maps $\u\colon \X\to \Y$ satisfying $\u(\min_{\X}(x))=\min_{\Y}(\u(x))$ for each $x\in X$.
For each $n\in\N$, let $\POn$ denote the full subcategory of $\POS$ whose objects $\X=(X,\leq)$ satisfy 
$|\min_{\X}(x)|\leq n$, for each $x\in X$.
Let $\POS_0$ denote the full subcategory of $\POS$ whose objects 
$\X=(X,\leq)$ satisfy $X=\min(\X)$.
For each $\X=(X,\leq)\in\POS$ and each $n\in\N$, let $(\X)_n=(X_n,\leq_{X_n})$ denote the subposet of $\X$ such that $X_n=\{x\in X\mid |\min_{\X}(x)|\leq n\}$. Let further $(\X)_0=(\min(\X),=)$. 
The assignment $\X\mapsto (\X)_n$ can be extended to a functor from $\POS$ to $\POn$ by mapping each morphism $\u\colon \X\to \Y$ to its restriction $(\u)_n=\u\restrict{X_n}$. 

 The categories $\PDL^f$ and $\POS$ are dually equivalent 
Let $\D\colon \PDL^f\rightarrow\POS$ and $\E\colon \POS\rightarrow \PDL^f$ denote the functors that determine that duality.
We omit the  detailed description of these functors, since it plays no role in the paper (see~\cite{Pr1975} or~\cite{Ur1981}). The only property of $\D$ and $\E$, that will find use in the paper is that for each $n\in\N\cup\{0\}$, their restrictions to the categories $\PDLnf$ and $\POn$ also determine a dual equivalence between these categories. 

%
%



\subsection*{Duals of projective p-lattices and unifiers}
%
%
Let $\X=(X,\leq)$ be a finite poset. 
Then $\X$ is said to satisfy condition 
\begin{itemize}
\item[$\Cstar$:] if for each $x,y\in X$ the least upper bound $x\vee_{\X} y$ of $x$ and $y$  exists in $\X$ and it satisfies 
$  \min_{\X}(x\vee_{\X} y)=\min_{\X}(x)\cup\min_{\X}(y)$.
\end{itemize}
For each $n\in\N$, the poset $\X$ is said to satisfy condition
\begin{itemize}
\item[$\Cnstar$:] if for each $x,y\in X$ such that $|\min_{\X}(x)\cup\min_{\X}(y)|\leq n$, the least upper bound $x\vee_{\X} y$  exists in $\pos{X}$ and satisfies 
$\min_{\X}(x\vee_{\X} y)=\min_{\X}(x)\cup\min_{\X}(y).$ 
\end{itemize}

It is easy to verify that $\X$ satisfies $\Cnstar$ if and only if $(\X)_n$ satisfies $\Cnstar$. Also observe that  $\X$ satisfies $\Cstar$ if and only if it satisfies $\Cnstar$ for each $n\in\N$.

\begin{theorem}\label{Theo:Proj}{\rm \cite{Ur1981}} Let $\alg{A}\in\PDL^f$. Then
\begin{itemize}
\item[(i)]$\alg{A}$ is projective in $\PDL$ if and only if $\D(\alg{A})$ is non-empty and satisfies condition $\Cstar$.
\item[(ii)] For each $n\in\N$, $\alg{A}$ is projective in $\PDLn$ if and only if $\D(\alg{A})$ belongs to $\POn$, $\D(\alg{A})$ is non-empty  and  satisfies condition $\Cnstar$. 
\end{itemize}
\end{theorem}

For later use we  define $(*_0)$: a finite poset $\X$ satisfies condition $(*_0)$ if it is non-empty. Since each non-trivial finite algebra in $\PDLZero$  is projective, we could replace $\N$ by $\N\cup\{0\}$ in Theorem~\ref{Theo:Proj}(ii) and the result will remain valid.

\begin{example}For each $m\in\N$, let $$
\pow{m}=(\pow{\{1,\ldots,m\}},\subseteq)
$$
be the poset of subsets of $\{1,\ldots,m\}$ ordered by inclusion and 
$$\Pm=\bigl(\{S\subseteq\{1,\ldots,m\}\mid S\neq\emptyset\},\subseteq\bigr).$$
The posets $\pow{m}$ and $\Pm$  are join-semilattices.  It is straightforward to check that $\min(\Pm)=\{\{1\},\ldots \{m\}\}$ and that $\pow{m}$ and $\Pm$ satisfy $\Cstar$.
Now Theorem~\ref{Theo:Proj} proves that $\E(\pow{m})$ and $\E(\Pm)$ are projective in  $\PDL$.
Observe that $S\subseteq\{1,\ldots,m\}$ is in $ (\Pm)_n$ if $1\leq|S|\leq n$. Then $(\Pm)_n\in \POn$ satisfy $\Cnstar$ and $\E((\Pm)_n)$ is projective in $\PDLn$. 
\begin{figure}
\begin{pspicture}(0,-.7)(12,4)
\psdots(1,1)(1,2)(2,1)(2,2)(3,1)(3,2)(2,3)(2,0)
\psline(2,0)(1,1)(1,2)(2,3)
\psline(2,0)(2,1)
\psline(2,2)(2,3)
\psline(2,0)(3,1)(3,2)(2,3)
\psline(1,1)(2,2)(3,1)
\psline(1,2)(2,1)(3,2)
\rput[c](2,-.5){(a) $\mathcal{P}(3)$}

\psdots(5,1)(5,2)(6,1)(6,2)(7,1)(7,2)(6,3)
\psline(5,1)(5,2)(6,3)
\psline(6,2)(6,3)
\psline(7,1)(7,2)(6,3)
\psline(5,1)(6,2)(7,1)
\psline(5,2)(6,1)(7,2)
\rput[c](6,-.5){(a) $\pos{P}(3)$}

\psdots(9,1)(9,2)(10,1)(10,2)(11,1)(11,2)
\psline(9,1)(9,2)
\psline(11,1)(11,2)
\psline(9,1)(10,2)(11,1)
\psline(9,2)(10,1)(11,2)
\rput[c](10,-.5){(a) $(\pos{P}(3))_2$}

\end{pspicture}
\caption{}\label{Fig:Pm}
\end{figure}
\end{example}

Combining the dualities between $\PDL$ and $\POS$, and between $\PDLn$ and $\POn$ with Theorem~\ref{Theo:Proj}, we  can translate the algebraic unification theory of $\PDL$ and $\PDLn$ into their dual categories as follows.
Let $\X\in \POS$. Then  $\Un{\POS}(\X)$ denotes the class of morphisms 
$\u\colon \Y\to \X $ with $\Y\in\POS$ satisfying $\Cstar$. For $\u \colon\Y\to \X,\nu\colon  \Z\to \X\in \Un{\POS}(\X)$, then we write $\u\leqn_{\POS} \nu$ if there exists a morphism $\psi\colon \Y\to \Z$ such that $\nu\circ \psi=\u$. 
For each $\alg{A}\in\PDL$ the preordered classes $(\Un{\POS}(\D(\alg{A})),\leqn_{\POS})$ and $(\Un{\PDL}(\alg{A}),\leqn_{\PDL})$ are categorically equivalent and have the same type.
Also observe that from Theorem~\ref{Theo:Proj}, it is easy to see that a $p$-lattice  admits a unifier if and only if it is non-trivial, equivalently, its dual poset is non-empty.

Similarly, for  $\X\in \POn$ we let $\Un{\smash \POn}(\X)$ denote the class of $p$-morphisms $\u$ from $\Y$ into $\X $ with $\Y\in\POn$ satisfying $\Cnstar$. The preordered classes $(\Un{\POn}(\D(\alg{A})),\leqn_{\POS})$ and $(\Un{\PDLn}(\alg{A}),\leqn_{\PDLn})$ are categorically equivalent and they have the same type.

In the rest of the paper we will use this translation and develop our results in the categories $\POS$ an $\POn$ instead of in $\PDL^f$ and $\PDLn$.

\section{Special product of finite posets}\label{Sec:SpecialProd}
In this section we introduce a construction in $\POS$ that preserves $\Cnstar$  and $\Cstar$ (in a sense that will be made clear in Theorem~\ref{Theo:OdotStar}). This construction posses certain properties (Theorems~\ref{Theo:OdotStar} and~\ref{Theo:Coproduct}) that will be used to study the unification type of posets in the rest of the paper.

Given finite posets $\X=(X,\leq_{\X})$ and $\Y=(Y,\leq_{\Y})$, we define $\X\odot\Y=(Z,\leq_{\Z})$ as follows:
$$
Z=
(X\times\{\bot\})\cup (X\times Y)\cup (\{\bot\}\times Y),
$$
where $\bot\notin X\cup Y$,
 and
$$
(x,y)\leq_{\Z} (x',y')\iff \bigl((x=\bot \mbox{ or } x\leq_{\X} x')\mbox{ and }(y=\bot \mbox{ or } y\leq_{\Y} y')\bigr).
$$
Clearly, $\X\odot\Y$ is the subposet of the product poset $(\{\bot\}\oplus \X)\times (\{\bot\}\oplus \Y)$ obtained by removing the element $(\bot,\bot)$, where $\{\bot\}\oplus \X$ and $\{\bot\}\oplus\Y$ are constructed by adding a fresh bottom element $\bot$ to  $\X$ and $\Y$, respectively.
It is easy to see that the maps $\iota_{\X}\colon \X \to \X\odot\Y$ and $\iota_{\Y}\colon \Y \to \X\odot\Y$ defined by $\iota_{\X}(x)=(x,\bot)$ and  $\iota_{\Y}(y)=(\bot,y)$ are $p$-morphisms, and that $\X$ and $\Y$ are isomorphic in $\POS$ to the subposets of $\X\odot\Y$ whose universes are $\iota_{\X}(X)$ and $\iota_{\Y}(Y)$, respectively.

\begin{example}\label{Ex:Eta}
Let $m,k\in\N$. Then $\pos{P}(m+k)$ and $\Pm\odot\pos{P}(k)$ are isomorphic in $\POS$. Indeed, let $\eta_{m,k}\colon\pos{P}(m+k)\to\Pm\odot\pos{P}(k)$ be the map defined by 
$$\eta_{m,k}(T)=\begin{cases}
(\bot, T') & \mbox{if }T\cap\{1,\ldots,m\}=\emptyset;\\
(T,\bot) & \mbox{if }T\cap\{m+1,\ldots,m+k\}=\emptyset;\\
(T\cap\{1,\ldots,m\},T')&\mbox{otherwise};
\end{cases}$$
 where $T'=\{i-m\mid i\in T\cap\{m+1,\ldots,m+k\}\}$.
 Then $\eta_{m,k}$ is a $p$-morphism and an isomorphism in $\POS$.
\end{example}

In the following theorems we present the  properties  of the construction $\X\odot\Y$ that we shall use in this paper.

\begin{theorem}\label{Theo:OdotStar}
Let $\X,\Y\in\POS$. Then 
\begin{itemize}
\item[(i)] $\min(\X\odot\Y)=\min(\X)\times\{\bot\}\cup\{\bot\}\times\min(\Y)$;
\item[(ii)] for each $x\in X$ and $y\in Y$;
\begin{align*}
\textstyle\min_{\X\odot\Y}(x,\bot)&\textstyle=\min_{\X}(x)\times\{\bot\},\\
\textstyle\min_{\X\odot\Y}(\bot,y)&\textstyle=\{\bot\}\times\min_{\Y}(y),\\
\textstyle\min_{\X\odot\Y}(x,y)&\textstyle=\min_{\X}(x)\times\{\bot\}\cup\{\bot\}\times\min_{\Y}(y);
\end{align*}
\item[(iii)] $(\X\odot\Y)_n$ satisfies $\Cnstar$ if and only if $\X$ and $\Y$ satisfy $\Cnstar$;
\item[(iv)]  $\X\odot\Y$ satisfies $\Cstar$ if and only if $\X$ and $\Y$ satisfy $\Cstar$.
\end{itemize}
\end{theorem}
\begin{proof}
The proofs of (i) and (ii) follow   from the fact that $(\bot,y),(x,\bot)\leq (x,y)$ for each $(x,y)\in X\times Y$.

To prove (iii) first assume that $\X$ and $\Y$ both satisfy $\Cnstar$. Let $(x,y),(x',y')\in\X\odot\Y$ be such that $|\min_{\X\odot\Y}(x,y)\cup\min_{\X\odot\Y}(x',y')|\leq n$.
By (ii), if  $x\neq\bot\neq x'$, then 
$|\min_{\X}(x)\cup\min_{\X}(x')|\leq |\min_{\X\odot\Y}(x,y)\cup\min_{\X\odot\Y}(x',y')|\leq n$. By $\Cnstar$, the least upper bound $x\vee_{\X}x'$ exists in $\X$ and $\min_{\X}(x)\cup\min_{\X}(x')=\min_{\X}(x\vee_{\X}x')$. The same argument applies when $y\neq\bot\neq y'$. Then we define
$$
s=\begin{cases}
\bot& \mbox{if }x=x'=\bot;\\
x&\mbox{if }x'=\bot\mbox{ and }x\neq\bot;\\
x'&\mbox{if }x=\bot\mbox{ and }x'\neq\bot;\\
x\vee_{\X}x'&\mbox{if }x'\neq\bot\neq x;
\end{cases}
\ \mbox{ and }\ 
t=\begin{cases}
\bot& \mbox{if }y=y'=\bot;\\
y&\mbox{if }y'=\bot\mbox{ and }y\neq\bot;\\
y'&\mbox{if }y=\bot\mbox{ and }y'\neq\bot;\\
y\vee_{\Y}y'&\mbox{if }y'\neq\bot\neq y.
\end{cases}
$$
Now it is tedious but straightforward to check that in each case the pair $(s,t)$ coincides with $(x,y)\vee_{\X\odot\Y}(x',y')$ and that $\min_{\X\odot\Y}(s,t)=\min_{\X\odot\Y}(x,y)\cup\min_{\X\odot\Y}(x',y')$.

The converse follows from the fact that $\X$ and $\Y$ are isomorphic to the subposets of ~$\X\odot\Y$ whose universes are $\iota_{\X}(X)$ and $\iota_{\Y}(Y)$, respectively. 
More precisely, let ${x,x'\in X}$ be such that $|\min{\X}(x)\cup\min_{\X}(x')|\leq n$. 
From (ii), it follows that $|\min_{\X}(x)\cup\min_{\X}(x')|=|\min_{\X\odot\Y}(x,\bot)\cup\min_{\X\odot \Y}(x',\bot)|\leq n$. 
Since  $\X\odot\Y$ satisfies $\Cnstar$,
 there exists $(u,v)\in \X\odot\Y$ such that $(x,\bot),(x',\bot)\leq (u,v)$ and satisfying $\min_{\X\odot\Y}(u,v)=\min_{\X\odot\Y}(x,\bot)\cup\min_{\X\odot \Y}(x',\bot)$. 
It follows that $x,y\leq u$, and $(x,\bot),(x',\bot)\leq (u,\bot)\leq (u,v)$. 
Therefore $\min_{\X\odot\Y}(u,\bot)=\min_{\X\odot\Y}(u,v)$, and, again by (ii), we conclude $\min_{\X}(u)=\min_{\X}(x)\cup\min_{\X}(x')$ 

The proof of (iv) follows directly from (iii).
\end{proof}


\begin{theorem}\label{Theo:Coproduct}
Let $\X_1,\X_2,\Y_1,\Y_2\in\POS$; $\u_1\colon\X_1\to\Y_1$ and $\u_2\colon\X_2\to\Y_2$ be $p$-morphisms. Then the map $\u_1\odot\u_2\colon \X_1\odot\X_2\to\Y_1\odot\Y_2$ defined by
$$
(\u_1\odot\u_2)(x,y)=\begin{cases}
(\u_1(x),\bot)&\mbox{ if } y=\bot;\\
(\bot,\u_2(y))&\mbox{ if } x=\bot;\\
(\u_1(x),\u_2(y))& \mbox{otherwise};
\end{cases}
$$
is a $p$-morphism {\rm (}not necessarily unique{\rm )} such that the diagram in Fig.~\ref{Fig:OdotMap} commutes.
\begin{figure} [ht]
\begin{center}
\begin{tikzpicture} 
[auto, 
 text depth=0.25ex,
] 
\matrix[row sep= .9cm, column sep= .9cm]
{
\node (X1) {$\X_1$}; &\node (X12) {$\X_1\odot\X_2$}; &\node (X2) {$\X_2$};\\
\node (Y1) {$\Y_1$}; &\node (Y12) {$\Y_1\odot\Y_2$}; &\node (Y2) {$\Y_2$};\\
};
\draw [->] (X1) to node {$\iota_{\X_1}$} (X12);
\draw [->] (X2) to node [swap] {$\iota_{\X_2}$} (X12);
\draw [->] (Y1) to node [swap] {$\iota_{\Y_1}$} (Y12);
\draw [->] (Y2) to node {$\iota_{\Y_2}$} (Y12);
\draw [->] (X1) to node [swap] {$\u_1$} (Y1);
\draw [->] (X2) to node {$\u_2$} (Y2);
\draw [->] (X12) to node {$\u_1\odot\u_2$} (Y12);

\end{tikzpicture}
\end{center}\caption{}\label{Fig:OdotMap}
\end{figure} 
\end{theorem}

\section{Unification type of subvarieties of $p$-lattices}\label{Sec:UnifType}

The main result in this section is stated in Theorem~\ref{The:ClasSubvar}, where we prove that the only non-trivial subvariety of $\PDL$ not having type $0$ is the variety of Boolean algebras. The latter is known to have type $1$, since each finitely presented (equivalently, finite) Boolean algebra is projective (see~\cite{MN1989}).

In \cite[Theorem 5.9]{Gh1997}, it is claimed that $\PDL$ has type $0$. The example  presented by the author is the 
poset $\pos{G}=(\{a,b,c,d,e,f\},\leq)$ 
(see Fig.~\ref{Fig:G}).
\begin{figure}[ht]
\begin{pspicture}(0,-1.20)(4,3.5)
\psdots(2,0)(1.3,1)(2.7,1)(1.3,2)(2.7,2)(2,3)
\psline(2,0)(1.3,1)(1.3,2)(2.7,1)(2.7,2)(2,3)(1.3,2)
\psline(2,0)(2.7,1)
\psline(1.3,1)(2.7,2)
\rput[c](2,-.3){$a$}
\rput[c](1,1){$b$}
\rput[c](3,1){$c$}
\rput[c](1,2){$d$}
\rput[c](3,2){$e$}
\rput[c](2,3.3){$f$}
\rput[c](2,-1){$\pos{G}$}
\end{pspicture}
\caption{}\label{Fig:G}
\end{figure}

In the mentioned theorem it is claimed that $\Un{\POS}(\pos{G})$ is directed. Even though the claim is correct, there is a small mistake in the proof. 
Given two maps $\u_1\colon\pos{Q}_1\to\pos{G}$ and $\u_2\colon\pos{Q}_2\to\pos{G}$ that are in $\Un{\POS}(\pos{G})$, a third map was constructed from $\u\colon\pos{R}\to\pos{G}$  where the poset $\pos{R}$ is the disjoint union of $\pos{Q}_1$ and $\pos{Q}_2$ with a new top element $\top$ and the map $\u$ is defined  by: $\u(x)=\upsilon_i(x)$ if $x\in Q_i$,  $\u(\top)=f$ is in $\Un{\pos{G}}$. 
The problem with this construction is that~$\pos{R}$ does not necessarily satisfy $\Cstar$, as the following example shows:
Let $\pos{Q}_1=\pos{Q}_2=\pos{P}(2)$. Then $\pos{R}$ as constructed above is ordered as in Fig.~\ref{Fig:BadEx}. Now observe that $\{1\}\vee_{\pos{R}}\{1'\}=\top$ and $\{1,1'\}\neq\min_{\pos{R}}(\top)$.
\begin{figure}[ht]
\begin{pspicture}(0,.3)(10,3.5)
\psdots(1,1)(2,1)(1.5,2)
\psdots(3,1)(4,1)(3.5,2)
\psdots(6,1)(7,1)(6.5,2)(8,1)(9,1)(8.5,2)(7.5,3)
\psline(1,1)(1.5,2)(2,1)
\psline(3,1)(3.5,2)(4,1)
\psline(6,1)(6.5,2)(7,1)
\psline(8,1)(8.5,2)(9,1)
\psline(6.5,2)(7.5,3)(8.5,2)
\rput[c](1,.5){$\{1\}$}
\rput[c](2,.5){$\{2\}$}
\rput[c](3,.5){$\{1'\}$}
\rput[c](4,.5){$\{2'\}$}
\rput[c](6,.5){$\{1\}$}
\rput[c](7,.5){$\{2\}$}
\rput[c](8,.5){$\{1'\}$}
\rput[c](9 ,.5){$\{2'\}$}
\rput[c](1.5,2.3){$\{1,2\}$}
\rput[c](3.5,2.3){$\{1',2'\}$}
\rput[c](6,2.3){$\{1,2\}$}
\rput[c](9,2.3){$\{1',2'\}$}
\rput[c](7.5,3.3){$\top$}

\end{pspicture}
\caption{}\label{Fig:BadEx}
\end{figure}

Nevertheless, the claims that $\Un{\POS}(\pos{G})$ is directed and has type $0$ are both true. It can be proved that $\Un{\POS}(\pos{G})$  is directed using the special product construction developed in Section~\ref{Sec:SpecialProd}.
In  Lemma~\ref{Lem:DirG} we present a slightly stronger result.

 Observe that the poset $\pos{G}$ is in $\POn$ for each $n\in\N$.

\begin{lemma}\label{Lem:DirG}
Let $\pos{G}$ be defined as  above. Then the preordered classes $\Un{\POS}(\pos{G})$ and $\Un{\POn}(\pos{G})$ for $n\in\N$ are directed.
\end{lemma}
\begin{proof}

We first prove that $\Un{\POS}(\pos{G})$ is directed.
Suppose $\u_1\colon\pos{Q}_1\to\pos{G}$ and $\u_2\colon\pos{Q}_2\to\pos{G}$ are in $\Un{\POS}(\pos{G})$. By Theorem~\ref{Theo:OdotStar}\,(iv), the poset $\pos{R}=\pos{Q}_1\odot\pos{Q}_2$ satisfies $\Cstar$. Let $\upsilon\colon\pos{R}\to\pos{G}$ be defined as follows:
$$
\u(x,y)=\begin{cases}
\upsilon_1(x)&\mbox{ if }y=\bot;\\
\upsilon_2(y)&\mbox{ if }x=\bot;\\
f&\mbox{otherwise}.
\end{cases}
$$
By Theorem~\ref{Theo:OdotStar}(ii), we have 
$\u(\min_{\pos{R}}(x,y))=\{a\}=\min_{\pos{G}}(\u(x,y))$ for each $(x,y)\in \pos{R}$. 
If $(x,y)\leq(x',y')$, then we have three cases:
\begin{itemize} 
\item[(a)] if  $(x',y')\in\pos{Q}_1\times\pos{Q}_2$ then $\u(x',y')=f\geq \u(x,y)$; 
\item[(b)] if $(x',y')\in \pos{Q}_1\times\{\bot\}$, then $(x,y)\in \pos{Q}_1\times\{\bot\}$ and $x \leq_{\pos{Q}_1}x'$, and  \[\u(x,y)=\u_1(x)\leq \u_1(x')=\u(x',y');\] 
\item[(c)] if $(x',y')\in \{\bot\}\times\pos{Q}_2$ the inequality $\u(x,y)\leq\u(x',y')$ follows from a routine variant of the argument used in case (b).
\end{itemize}
This proves that $\u\in\Un{\POS}(\pos{G})$.
By definition of $\u$, it follows that $\iota_{\pos{Q}_i}\colon\pos{Q}_i\to \pos{R}$ satisfies $\u\circ \iota_{\pos{Q}_i}=\u_i$ for each $i\in\{1,2\}$. 
Having thus proved  $\u_1,\u_2\leqn \u$, we  conclude that $\Un{\POS}(\pos{G})$ is directed.

The proof that  $\Un{\POn}(\pos{G})$ is directed follows by a similar construction using Theorem~\ref{Theo:OdotStar}\,(iii) and defining $\pos{R}=(\pos{Q}_1\odot\pos{Q}_2)_n$.
\end{proof}
We will now determine the unification type of each subvariety of~$\PDL$.
\begin{theorem}\label{The:ClasSubvar}
Let $\variety{V}$ be a non-trivial subvariety of $ \PDL$. Then the following holds:
$$
\Type{}(\variety{V})=
\left\{
\begin{tabular}{ll}
$1$&  if $\variety{V}=\PDLZero$; \\
$0$ & otherwise.
\end{tabular}
\right.
$$
\end{theorem}
\begin{proof}
The variety $\PDLZero$ is the class of Boolean algebras. Since every non-trivial finitely presented Boolean algebra is projective, $\PDLZero$ i has type  $1$. We conclude that $\PDLZero$ has type~$1$.

Combining Lemma~\ref{Lem:DirG} with the argument in  \cite[Theorem 5.9]{Gh1997}, we obtain that $\Type{}(\Un{\POS}(\pos{G}))=0$. Therefore~$\PDL$ has type $0$.

Now let us fix $n\in\N$. We will use the same construction used in \cite{Gh1997} to prove that each~$\PDLn$ has type~$0$.
For each $m\in\N$, let  $\u_m\colon\mathcal{P}(m)\to\pos{G}$ be the map defined by
$$
\u_m(S)=\begin{cases}
a&\mbox{ if } S=\emptyset;\\
b&\mbox{ if } S=\{k\}\mbox{ for some even }1\leq k\leq m;\\
c&\mbox{ if } S=\{\ell\}\mbox{ for some odd }1\leq \ell\leq m;\\
d&\mbox{ if } S=\{k,\ell\}\mbox{ for some }1\leq k<\ell\leq m\\
&\ \ \mbox{ such that }k\mbox{ is even and }\ell\mbox{ is odd};\\
e&\mbox{ if } S=\{k,\ell\}\mbox{ for some }1\leq k<\ell\leq m\\
&\ \ \mbox{ such that }k\mbox{ is odd and }\ell\mbox{ is even};\\
f&\mbox{otherwise}.
\end{cases}
$$

Clearly $\mathcal{P}(m)\in \POn$ and $\mathcal{P}(m)$ satisfy $\Cnstar$ for each $m\in\N$. From the definition above, it follows that $\u_{m}$ is monotone for each $m\in\N$. Since $\pos{G}$ has only one minimal element, it follows that each $\u_{m}$ is a $p$-morphism and therefore in $\Un{\POn}(\pos{G})$. It is easy to observe that for each $m\in \N$ the inclusion map $\varphi_{m}$ from $\mathcal{P}(m)$ into $\mathcal{P}(m+1)$ is such that $\u_{m+1}\circ \varphi_m=\u_m$, hence $\u_m\leqn  \u_{m+1}$.
 
Suppose that $\u\colon\Y\to\pos{G}$ is a $p$-morphism such that $\Y$ satisfies $\Cnstar$ 
and that there exists a morphism $\nu\colon\mathcal{P}(m)\to \Y$ in $\POS_n$ such that $\u \circ \nu=\u_m$.

We claim that for each  $i,j\in\{1,\ldots, m\}$, 
if $\nu(\{i\})=\nu(\{j\})$, then $i=j$.
By way of contradiction assume that $i<j$. Since  $\u_m(\{i\})=\u\circ\nu(\{i\})=\u\circ\nu(\{j\})=\u_m(\{j\})$,  both $i$ and $j$ have the same parity. Let $k$ be  such that $i<k<j$ and the parity of $k$ is different from that of $i$ and $j$. 
Then 
$$
  \{b,c\}=\{\u_m(\{i\}),\u_m(\{k\})\}=\{\u(\nu(\{i\})),\u(\nu(\{k\}))\}$$
and
\[  \{\u(\nu(\{i,k\})),\u(\nu(\{k,j\}))\}=\{\u_m(\{i,k\}),\u_m(\{k,j\})\}
=\{e,d\}.
\]

 Since $\Y$ satisfies $\Cnstar$ 
and 
\begin{align*}
\textstyle|\min_{\Y}(\nu(\{i\}))\cup\min_{\Y}(\nu(\{k\}))|&\textstyle=|\nu(\min_{\mathcal{P}(m)}(\{i\}))\cup \nu(\min_{\mathcal{P}(m)}(\{k\}))|\\
&=|\nu(\emptyset)|=1,
\end{align*}
 the least upper bound $x=\nu(\{i\})\vee_{\Y}\nu(\{k\})$ exists. 
Using the fact that $\nu$, $\u$ and $\u_m$ are order preserving, we have that $b,c\leq\u(x)\leq e,d$.
This contradicts the fact that there does not exist an element $y\in\pos{G}$ such that $b,c\leq y\leq e,d$. 

From this we obtain that
$$
m=|\{\nu(\{i\})\mid i\in\{1,\ldots,m\}\}|\leq|\Y|.
$$ 
Therefore, a common upper bound to the sequence $\u_1\leqn\u_2 \leqn\cdots $ should have an infinite domain.
As a consequence, there does not exist an upper bound in $\Un{\POn}(\pos{G})$ for the sequence $\u_1\leqn\u_2 \leqn\cdots $. 
From Lemma~\ref{Lem:DirG} and Theorem~\ref{Theo:Baader}\,(ii), it follows that the type of  $\Un{\POn}(\pos{G})$ is $0$. Therefore, $\Type{}(\PDLn)=0$.
\end{proof}

\section{Type of unification problems in 
$\mathfrak{B}_1$
}\label{Sec:UnifPDL1}

We already have the machinery to present the classification of the unification problems in $\PDLOne$. This will serve as a warm-up for the analysis of unification types in $\PDL$ and $\PDLn$  in Sections~\ref{Sec:UnifPDL} and~\ref{Sec:UnifPDLn}, respectively. Even though the results in this section are less technically involved  than the ones presented in Sections~\ref{Sec:UnifPDL} and~\ref{Sec:UnifPDLn}, the structure of these sections is similar. Initially, we present some necessary conditions for a poset in $\POS_1$ to have unification type $0$ (Lemma~\ref{Lem:POne0}). Finally, we determine the type of each poset in $\POS_1$ (Theorem~\ref{Theo:MainPone}), depending on its properties, by presenting a minimal  complete set of unifiers, or using Lemma~\ref{Lem:POne0} to see that it has type $0$.

\begin{lemma}\label{Lem:POne0}
Let $\X\in \POS_1$. If there exist $a,b,c,d,x\in X$ such that
\begin{itemize}
\item[(i)] $a,b\leq c,d\leq x$;
\item[(ii)] there is no $e\in X$ such that $a,b\leq e\leq c,d$,
\end{itemize}
then $\Type{}(\Un{\POS_1}(\X))=0$.
\end{lemma}
\begin{proof}
By Theorem~\ref{Theo:OdotStar}(iii), if $\Y_1$ and $\Y_2$ are in $\POS_1$ and satisfy $(*_1)$, then $(\Y_1\odot\Y_2)_1$ satisfies $(*_1)$. Observe that $(\Y_1\odot\Y_2)_1$ is isomorphic in $\POS_1$ to the disjoint union of $\Y_1$ and $\Y_2$. It follows that the class $\Un{\POS_1}(\X)$ is directed.

Since $\X\in \POS_1$, there exists $y\in X$ such that $\min_{\X}(x)=\{y\}$. Condition (ii) implies that $|\{a,b, c,d\}|=4$. Then from condition (i) it follows that the subposet $\Z=(\{y,a,b,c,d,x\},\leq)$ of $\X$ is isomorphic to the poset $\pos{G}$ shown in Fig.~\ref{Fig:G}.
Accordingly, we are in position to define a sequence of unifiers $\u_m\colon\mathcal{P}(m)\to  \X$ with image contained in $\Z$ as in Theorem~\ref{The:ClasSubvar}. Applying the same arguments used in the proof of Theorem~\ref{The:ClasSubvar}, we can show that if $\u\colon\Y\to \X$ is such that $\u_m\leqn\u$, then $m\leq|Y|$. 

Now, an application of  Theorem~\ref{Theo:Baader}\,(ii) proves that the type of $\Un{\POS_1}(\X)$ is $0$.
\end{proof}
\begin{theorem}\label{Theo:MainPone}
Let $\X$ be a non-empty poset in $\POS_1$. Then 
$$
\Type{}(\Un{\POS_1}(\X))=\begin{cases}
1& \mbox{ if } \DwSet{x} \mbox{ is a lattice for each } x\in \X;\\
0 & \mbox{otherwise}.
\end{cases}
$$
\end{theorem}
\begin{proof}
Assume first, that for each $x\in X$ the set $\DwSet{x}$ with the inherited order is a lattice.
Consider the set $R=\{(x,y)\mid x\in X\mbox{ and }y\leq x\}$ ordered by $(x,y)\leq(x',y')$ if $x=x'$  and $y\leq y'$.
It follows that $\min_{\pos{R}}(x,y)=\{(x,m_x)\}$ for any $(x,y)\in R$,  where $m_x$ is the unique element in $\min_{\X}(x)$.
If $(x,y),(x',y')\in R$ are such that  $|\min_{\pos{R}}(x,y)\cup\min_{\pos{R}}(x',y')|=1$, then $x=x'$, so $y,y'\leq x$.
 By assumption, there exists $y\vee_{\DwSet{x}}y'$ and $\min_{\X}(y\vee_{\DwSet{x}}y')=\min_{\X}(x)=\{m_x\}$. Then the element $(x,y\vee_{\DwSet{x}}y')$ is the least  upper bound of $(x,y)$ and $(x',y')$ in $\pos{R}$ and satisfies $\min_{\pos{R}}(x,y\vee_{\DwSet{x}}y')=\{(x,m_x)\}=\min_{\pos{R}}(x,y)\cup\min_{\pos{R}}(x,y')$. This proves that $\pos{R}$ satisfies $(*_1)$.

Let the map
$\eta\colon \pos{R}\to \X$ be defined by $\eta(x,y)=y$. Then $\eta\in\Un{\smash {\POS_1}}(\X)$. 
We claim that $\{\eta\}$ is a minimal complete set for $\Un{\POS_1}(\X)$.
Since $\Y$ is in $\POS_1$ and satisfies $(*_1)$ for each $y\in\Y$, there exists a unique $M_y\in\max(\Y)$ such that $y\leq M_y$. Indeed, since $\Y\in\POS_1$, if $z,z'\in \Y$ are such that $y\leq z,z'$, then $\min_{\Y}(z)=\min_{\Y}(z')=\min_{\Y}(y)$ and $|\min_{\Y}(z\cup z')|=1$. Now from $(*_1)$, there exists $z\vee_{\Y}z'$. Therefore $\UpSet{y}$ is a finite join-semilattice, hence it has a maximal element $M_y$.
Let $\u\colon\Y\to \X\in\Un{\POS_1}(\X)$. Let $\nu\colon\Y\to\pos{R}$ be defined by $\nu(y)=(\u(M_y),\u(y))$. 
It is not hard to see that $\eta\circ\nu=\u$ and that $\nu$ is a morphism in $\POS_1$.
 Then $\u\leqn \eta$, which proves $\Type{\POS_1}(\X)=1$.

Suppose now that for some $x\in\X$ the set $\DwSet{x}$ with the inherited order from $X$ is not a lattice. That is, there exist two elements in $\DwSet{x}$ that do not have a least upper bound or a greatest lower bound. In each of these cases, there are $a,b,c,d\in\DwSet{x}$ such that 
\begin{itemize}
\item[(i)] $a,b\leq c,d$;
\item[(ii)] there is no $e\in X$ such that $a,b\leq e\leq c,d$.
\end{itemize}
From Lemma~\ref{Lem:POne0}, it follows that $\Type{}(\Un{\POS_1}(\X))=0$.
\end{proof}

\section{Connected sets}\label{Sec:Unifcore}

In this section we introduce two key notions: connected set and $n$-connected set. These concepts will play a central role in our description of the unification types in Sections~\ref{Sec:UnifPDL} and~\ref{Sec:UnifPDLn}.

\begin{definition}\label{Def:Con}
Let $\X=(X,\leq)\in\POS$ and $Y\subseteq X$. We say that $Y$ is 
\emph{connected} if it satisfies
\begin{itemize}
\item[(i)] $\min_{\X}(Y)\subseteq Y$;
\item[(ii)] for each $x,y\in Y$ there exists $z\in Y$ such that $x,y\leq z$ and \[\textstyle\min_{\X}(x)\cup\min_{\X}(y)=\min_{\X}(z).\]
\end{itemize}
\end{definition}
Let $\Con(\X)$ denote the poset of connected subsets of $\X$ ordered by inclusion. 
Observe that (i) in Definition~\ref{Def:Con} implies that $\min_{\Y}(y)=\min_{\X}(y)$, for each $Y\in\Con(\X)$ and each $y\in Y$.

For later use in Theorem~\ref{Theo:ClassPDL}, we collect here some properties of $\Con(\X)$. The first follows directly from the definition.
\begin{lemma}\label{Lem:ExtCon}
Let $\X\in\POS$ and $Y\in \Con(\X)$. If $x\in X$ and $y\in Y$ satisfy $x\leq y$ and $\min_{\X}(x)=\min_{\X}(y)$, then $Y\cup\{x\}\in\Con(\X)$.
\end{lemma}
\begin{lemma}\label{Lem:RangCon}
Let $\X,\Y\in\POS$ and $\u\colon \X\to \Y$ be a $p$-morphism. If $\X$ satisfies $\Cstar$  then
  $\u(\X)\subseteq\Y$ is connected.
\end{lemma}
\begin{proof}
Since  $\u$ is a $p$-morphism, it follows that $\u(\X)$ satisfies (i) in Definition~\ref{Def:Con}.
Let $x,y\in \u(X)$ and $x',y'\in X$ be such that $\u(x')=x$ and $\u(y')=y$. Then \begin{align*}
\textstyle\min_{\Y}(\u(x'\vee_{\X}y'))&\textstyle=\u(\min_{\X}(x'\vee_{\X} y' ))=\u(\min_{\X}(x')\cup\min_{\X}( y' ))\\
&\textstyle=\u(\min_{\X}(x'))\cup\u(\min_{\X}( y' ))=\min_{\Y}(x)\cup\min_{\Y}(y).
\end{align*}
This proves that $\u(X)$ satisfies item (ii) of Definition~\ref{Def:Con}.
\end{proof}

The poset of connected subsets will be used in our description of the unification type of a poset in $\POS$ (Theorem~\ref{Theo:ClassPDL}). 
To study the unification type in $\POn$ (Section~\ref{Sec:UnifPDLn}) we need the slightly more sophisticated notion of $n$-connected set. 

\begin{definition}\label{Def:n-con}
Let $\X=(X,\leq)\in\POn$ and $Y\subseteq X$. We say that $Y$ is 
\emph{$n$-connected} if it satisfies
\begin{itemize}
\item[(i)] $\min_{\X}(Y)\subseteq Y$;
\item[(ii)] for each $S\subseteq \min_{\X}(Y)$ such that $|S|\leq n$, there exists $z\in Y$ satisfying $\min_{\X}(z)=S$;
\item[(iii)] if $x,y\in Y$ satisfy $\min_{\X}(x)=\min_{\X}(y)$ and $|\min_{\X}(x)|< n$, then there exists a sequence $x_0,x_1,\dots,x_{r-1},x_{r}\in Y$ such that $$x=x_0\geq x_1\leq x_2\geq x_3\leq\ \ldots\ \geq x_{r-1}\leq x_{r}=y$$
and $\min_{\X}(x_i)=\min_{\X}(x)$ for each $0\leq i\leq r$.
\end{itemize}
\end{definition}
Let  $\Cn(\X)$ denote the poset of $n$-connected subsets of $\X$ ordered by inclusion. It is easy to observe that $\Con(\X)\subseteq\Cn(\X)$ for each $\X\in\POn$. 

 We now collect some properties of $n$-connected sets.
The proof of the first lemma follows directly from the definition of $n$-connected set.

\begin{lemma}\label{Lem:UpSetNCon}
Let $\X$ be a poset in $\POn$ and $Y\in\Cn(\X)$. If $x\in X$ and $y\in Y$ satisfy $\min_{\X}(x)=\min_{\X}(y)$ and ($x\leq y$ or $y\leq x$), then $Y\cup\{x\}\in\Cn(\X)$.
\end{lemma}
\begin{lemma}\label{Lem:RangN-Con}
Let $\X,\Y\in\POn$.  Let $\u\colon \X\to \Y$ be a $p$-morphism. If $\X$ satisfies $\Cnstar$  then
  $\u(\X)\subseteq\Y$ is $n$-connected.
\end{lemma}
\begin{proof}
Since $\u$ is a $p$-morphism, it follows that $\u(\X)$ satisfies Definition~\ref{Def:n-con}(i).  Condition (ii) follows directly from the fact that $\X$ satisfies $\Cnstar$.

To prove item (iii), let $x,y\in \u(X)$ be such that $\min_{\X}(x)=\min_{\X}(y)$ and both sets have cardinality $k< n$. 
There exist $x',y'\in X$ such that $\u(x')=x$ and $\u(y')=y$. Since $\u$ commutes with $\min$, there exist $S\subseteq \min_{\X}(x')$ and $T\subseteq\min_{\X}(y')$ such that $|S|=|T|=k<n$, $\u(S)=\min_{\Y}(x)=\min_{\Y}(y)=\u(T)$.

Let $s_1,\ldots s_k$ and $t_1,\ldots,t_k$ be enumerations of $S$ and $T$, respectively, such that $\u(s_i)=\u(t_i)$ for each $1\leq i \leq k$.
For each $l\in\{1,\ldots,k\}$, the elements $$\textstyle y_{l}=\bigvee_{\X}\{s_i,t_j\mid i\geq l > j\}\mbox{ and }z_{l}=\bigvee_{\X}\{s_i,t_j\mid i\geq l \geq j\}$$ 
 are well-defined, since $|\{s_i,t_j\mid i\geq l > j\}|\leq|\{s_i,t_j\mid i\geq l \geq j\}|\leq k+1\leq n$ and $\X$ satisfies $\Cnstar$. 
Let also $y_{k+1}$ be equal to $\bigvee_{\X}\{t_j\mid  j\in\{1,\ldots,k\}\}$.
Then for each $l\in\{1,\ldots,k\}$, we have  $y_l\leq z_{l}\geq y_{l+1}$. Thus
$$
x\geq y_1\leq z_1\geq y_2\leq z_2 \geq \ \ldots\ \geq y_{k+1}\leq y 
$$
and applying $\u$, we obtain the sequence
$$
\u(x)\geq\u(y_1)\leq  \u(z_1)\geq \u(y_2)\leq \u(z_2) \ldots \u(z_k)\geq\u(y_{k+1})\leq \u(y).$$ 
Since $\u(s_i)=\u(t_i)$ for each $1\leq i \leq k$ and $\u$ is a $p$-morphism, 
\[
\textstyle\min_{\Y}(\u(y_i))=\min_{\Y}(\u(z_i))=
\min_{\Y}(\u(y_{k+1}))=\min_{\Y}(x)=\min_{\Y}(y)\] for each $i\in\{1.\ldots,k\}$. 
This finishes the proof that $\u(X)$ satisfies condition (iii) of Definition~\ref{Def:n-con}.
\end{proof}

\section{Type of unification problems in~$\mathfrak{B}_{\omega}$
}\label{Sec:UnifPDL}
In Theorem~\ref{Theo:ClassPDL},  we present a description of the type of unification problems in~$\PDL$. As in Section~\ref{Sec:UnifPDL1}, using the duality between finite $p$-lattices and finite posets, the result is presented in terms of  unification type of finite posets. 

This section is structured as follows. In Lemma~\ref{Lem:PDL}, we prove that given $\X\in\POS$, if is there exists a maximal connected subset $Y$ of $\X$ that does not satisfy $\Cstar$, then the type of $\Un{\POS}(\X)$ is $0$. The proof of this lemma splits into three cases  which are developed separately in Lemmas~\ref{Lem:PDLA}, \ref{Lem:PDLB} and \ref{Lem:PDLC}. Finally in Theorem~\ref{Theo:ClassPDL}, we give the unification type of each poset in $\POS$.

\begin{lemma}\label{Lem:PDLA}
Let $\X$ be a finite poset. Assume there is $Y\in\max(\Con(\X))$,  $a,b,c,d\in Y$ and $n\in \N$ satisfying the following conditions
\begin{itemize}
\item[(i)] $\Y$ satisfies $(*_{n-1})$;
\item[(ii)] $|\min_{\Y}(a)\cup\min_{\Y}(b)|=n;$
\item[(iii)] $\min_{\Y}(a)\not\subseteq\min_{\Y}(b)$ and $\min_{\Y}(b)\not\subseteq\min_{\Y}(a)$;
\item[(iv)] $\min_{\Y}(c)=\min_{\Y}(d)=\min_{\Y}(a)\cup\min_{\Y}(b)$;
\item[(v)] $a,b\leq c,d$;
\item[(vi)] there is no $e\in Y$ such that $a,b\leq e\leq c,d$.
\end{itemize}
Then $\Type{}(\Un{\POS}(\X))=0$.
\end{lemma}
\begin{proof}

By (ii) and (iii) there exists an enumeration $x_1,\ldots,x_n$   of the elements of $\min_{\Y}(a)\cup\min_{\Y}(b)$ such that $x_{n-1}\in\min_{\Y}(a)\setminus\min_{\Y}(b)$ and $x_{n}\in\min_{\Y}(b)\setminus\min_{\Y}(a)$. Since $Y$ is connected, there exists $z\in Y$ such that $c,d\leq z$ and   \begin{equation}\label{Eq:MinZ}
\textstyle
\min_{\Y}(z)=\min_{\Y}(c)\cup\min_{\Y}(d)=\{x_1,\ldots,x_n\}.
\end{equation}

Now let $f\colon \N\to \{x_1,\ldots,x_n\}$ be the map defined by 
$$
f(i)=\begin{cases}
x_i& \mbox{if }i\leq n;\\
x_{n-1}&\mbox{if } i> n\mbox{ and } i\mbox{ is odd};\\
x_n&\mbox{if } i> n\mbox{ and } i\mbox{ is even}.\\
 \end{cases}
$$

By (i), for each  $S\subseteq Y$ such that $|\min_{\Y} S|< n$, the supremum $\bigvee_{\Y} S$ exists in $\Y$ and $\min_{\Y}(\bigvee_{\Y}S)= \min_{\Y} (S)$. 
Hence, for each $m\in\N$ the map $\u_{m}\colon\pos{P}(m)\to \Y$ given~by:
\[
\u_m(T)=\begin{cases}
\bigvee_{\Y}f(T)&\mbox{if }\min_{\Y}(a)\not\subseteq f(T) \mbox{ and }\min_{\Y}(b)\not\subseteq f(T);\\
a\vee_{\Y}\bigvee_{\Y}f(T)&\mbox{if }\min_{\Y}(a)\subseteq f(T)\mbox{ and }f(T)\neq\{x_1,\ldots,x_n\};\\
b\vee_{\Y}\bigvee_{\Y}f(T)&\mbox{if }\min_{\Y}(b)\subseteq f(T)\mbox{ and }f(T)\neq\{x_1,\ldots,x_n\};\\
c&\mbox{if }T=\{1,\ldots,n-2\}\cup\{i,j\};\\
&\ \mbox{ with }n-2<i\leq j\mbox{ and }i\mbox{ odd and }j\mbox{ even}; \\
d&\mbox{if }T=\{1,\ldots,n-2\}\cup\{i,j\};\\
&\ \mbox{ with }n-2<j\leq i\mbox{ and }i\mbox{  odd and }j\mbox{ even}; \\
z&
\mbox{otherwise},
\end{cases}
\]
 is well defined. It is straightforward from (v) and the fact that $c,d\leq z$ that each $\u_m$ is order preserving. By (iv) and \eqref{Eq:MinZ}, $\min_{\X}(\u_m (T))=\min_{\Y}(\u_m (T))=f(T)$. Therefore,  each $\u_m$ is a $p$-morphism. 

For each $m\in \N$, let $\varepsilon_{m}\colon \Pm\to\pos{P}(m+1)$ be the inclusion map. It is easy to check that $\u_m=\u_{m+1}\circ \varepsilon_{m}$. It follows that $\u_1\leqn \u_2\leqn  \cdots$.

Let $\u\colon\Z\to \X$ in $\Un{\POS}(\X)$ such that $\u_{m}\leqn \u$ for some $m\geq n$. We claim:
\begin{itemize}
\item[(a)]  $|Z|\geq m$; and 
\item[(b)] there exists $\nu\in\Un{\POS}(\X)$ such that $\u_{m+1},\u\leqn \nu $.
\end{itemize}

Before proving our claims let us fix $\psi\colon \Pm\to \Z$ a $p$-morphism such that ${\u\circ \psi=\u_m}$. 

Assume now that (a) does not hold, that is, $|Z|<m$. Necessarily,  there exist $i,j\in \{n-1,\dots,m\}$ such that $\psi(\{1,\ldots,n-2\}\cup\{i\})=\psi(\{1,\ldots,n-2\}\cup\{j\})$. By the definition of $f$ and $\u_m$, the numbers $i$ and $j$ have the same parity. Without loss of generality, assume $i$ and $j$ are odd and $i<j$.
Let $k$ be an even number such that $i<k<j$. Then $x=\psi(\{1,\ldots,n-2\}\cup\{i\})\vee_{\Z}\psi(\{1,\ldots,n-2\}\cup\{k\})$ is such that 
$a,b\leq \u(x)\leq c,d$. By (vi), $\u(x)\notin \Y$ and by Lemma~\ref{Lem:ExtCon}, it follows that $\Y\cup\{\u(x)\}$ is connected, which contradicts the maximality of $Y$ in $\Con(\X)$.
This concludes the proof of (a).

To prove (b), assume first that $m$ is odd.
Let $u\in\Z$, if  the set $$\{j\in\{n-1\ldots,m\}\mid j\mbox{ is even and } \psi(\{j\})\not\leq u\}$$ is non-empty, let \[j_u=\max(\{j\in\{n+1\ldots,m\}\mid j\mbox{ is even and } \psi(\{j\})\not\leq u\}).\]
 In this case let  \[\textstyle
u'=u\vee_{\Z}\bigvee_{\Z}\{\psi(S\cup\{j_u\})\mid S=\emptyset\mbox{ or }(S\in\pos{P}(m)\mbox{ and  }\psi(S)\leq u)\}.\]

Now, let $\nu\colon \Z\odot\POne\to\X$  be the map defined by:
$$
\nu(u,S)=\begin{cases}
\u(u)&\mbox{if }S=\bot;\\
f(m+1)&\mbox{if }u=\bot;\\
\u(u')&\mbox{if }u\neq\bot$, $S=\{1\}$, \mbox{ and }$\psi(\{j\})\not\leq_{\Z} u\\
&\ \ $for some even $j\in\{n-1,\ldots,m\};\\
\u(u)&\mbox{if }u\neq\bot$, $S=\{1\}$ and $\psi(\{j\})\leq_{\Z} u\\
&\ \  $for each even $j\in\{n-1,\ldots,m\}.\\
\end{cases}
$$
The equality $\nu(\min_{\Z\odot\POne}(u,S))=\min_{\X}(\nu(u,S))=\min_{\Y}(\nu(u,S))$ follows directly from the definition of $\nu$ and $u'$.
 Let $(u,S),(w,T)\in\Z\odot\POne$ be such that $(u,S)\leq(w,T)$. If $S=\bot$, then $u,w\in\Y$, and  $\nu(u,S)=\u(u)\leq \u(w)\leq \nu(w,T)$. If $S=\{1\}$ and $u\neq\bot$ then $\nu(u,S)=\u(u')\leq \u(w')=\nu(w,T)$. Finally, if $u=\bot$ and $w\in \Y$, then $\nu(u,S)=f(m+1)=f(m-1)=\u(\psi(\{m-1\}))\leq \u(w\vee_{\Z}\psi(\{m-1\}))$. Since $m$ is odd, $m-1$ is even and $w\vee_{\Z}\psi(\{m-1\})\leq v'$. Therefore, \[\nu(u,S)\leq \u(w\vee_{\Z}\psi(\{m-1\}))\leq\nu(w,T).\] 
We conclude that  $\nu$ is order preserving and a $p$-morphism.

From the definition of $\nu$ it follows that $\nu\circ\iota_{\Y}=\u$. This implies that $\u\leqn \nu$.

 We shall now prove that $\u_{m+1}\leqn \nu$. We claim $\nu\circ(\psi\odot{\rm Id}_{\POne})\circ\eta_{m,1}=\u_{m+1}$, where $\eta_{m,1}\colon\pos{P}(m+1)\to \pos{P}(n)\odot\POne $ is  defined in Example~\ref{Ex:Eta}, and ${\rm Id}_{\POne}\colon\POne\to\POne$ denotes the identity map. Let $T\in \pos{P}(m+1)$. If $m+1\notin T$, then 
\begin{align*}
(\nu\circ(\psi\odot{\rm Id}_{\POne})\circ\eta_{m,1})(T)&=
(\nu\circ(\psi\odot{\rm Id}_{\POne}))(T,\bot)\\
&=\nu(\psi(T),\bot)=\u_{m}(T)=\u_{m+1}(T).
\end{align*} 
 
If $T=\{m+1\}$, $$(\nu\circ(\psi\odot{\rm Id}_{\POne})\circ\eta_{m,1})(T)=\nu(\bot,\{1\})=f(m+1)=\nu_{m+1}(\{m+1\}).$$
If $m+1\in T$ and $T\setminus\{m+1\}\neq \emptyset$, let $S=\{j\in\{n-1,n,\ldots,m\}\mid j\mbox{ even and } j\notin T\}$. If $S=\emptyset$,  there are at least two even elements in $\{n-1,n,\ldots,m\}\cap T$ and define  $T'= T\setminus\{m+1\}$. In case $S\neq\emptyset$ define $i=\max(S)$ and $T'=\{i\}\cup T\setminus\{m+1\}$. Whether or not $S$ is empty, we can write \begin{align*}
(\nu\circ(\psi\odot{\rm Id}_{\POne})\circ\eta_{m,1})(T)&=
(\nu\circ(\psi\odot{\rm Id}_{\POne}))(T\setminus\{m+1\},\{1\})\\
&=\nu(\psi(T\setminus\{m+1\}),\{1\})
=\u_{m}(T')\\
&=\u_{m+1}(T).
\end{align*} 

We have proved (b) when $m$ is odd. Replacing odd for even and vice versa in the previous argument, we obtain a proof of (b) for $m$ even.

Finally from (a) it follows that $\u_1,\u_2,\ldots$ do not have a common upper bound. This, combined with (b), proves that the  sequence $\u_1,\u_2,\ldots$ satisfies condition~(i) of Theorem~\ref{Theo:Baader}. Therefore, $\Type{}(\Un{\POS}(\X))=0$.
\end{proof}

The reader will notice that the statements of the next two lemmas can be simplified. We have chosen not to do so to highlight the similarities between them and Lemma~\ref{Lem:PDLA}.  In this way, all these statements differ only on condition (iii). This stresses the point that Lemmas~\ref{Lem:PDLA}, \ref{Lem:PDLB} and \ref{Lem:PDLC} are particular cases needed in the proof of Lemma~\ref{Lem:PDL}. 
\begin{lemma}\label{Lem:PDLB}
Let $\X$ be a finite poset. Assume there is $Y\in\max(\Con(\X))$,  $a,b,c,d\in Y$ and $n\in \N$ satisfying the following conditions\begin{itemize}
\item[(i)] $\Y$ satisfies $(*_{n-1})$;
\item[(ii)] $|\min_{\Y}(a)\cup\min_{\Y}(b)|=n;$
\item[(iii)] $\min_{\Y}(a)=\min_{\Y}(b)$;
\item[(iv)] $\min_{\Y}(c)=\min_{\Y}(d)=\min_{\Y}(a)\cup\min_{\Y}(b)$;
\item[(v)] $a,b\leq c,d$;
\item[(vi)] there is no $e\in Y$ such that $a,b\leq e\leq c,d$.
\end{itemize}
Then $\Type{}(\Un{\POS}(\X))=0$.
\end{lemma}
\begin{proof}
Let $x_1,\ldots,x_n$ an enumeration of the elements of $\min_{\Y}(a)=\min_{\Y}(b)$.
Let $f\colon \N\to \{x_1,\ldots,x_n\}$ be defined by 
$$
f(i)=\begin{cases}
x_i& \mbox{if }i\leq n\\
x_{n}&\mbox{if } i> n.\\
 \end{cases}
$$
Since $Y$ is connected, there exits $z\in Y$ such that $c,d\leq z$

For each $m\in\N$ we define the maps $\u_{m}\colon\pos{P}(m)\to \X$ as follows:
$$
\u_m(T)=\begin{cases}
\bigvee_{\Y}f(T)&\mbox{if }f(T)\neq\{x_1,\ldots,x_n\};\\
a&\mbox{if } T=\{1,\ldots,n-1\}\cup\{i\} \mbox{ and } i\geq n\mbox{  is odd;}\\
b&\mbox{if } T=\{1,\ldots,n-1\}\cup\{i\} \mbox{ and } i\geq n\mbox{  is even;}\\
c&\mbox{if }T=\{1,\ldots,n-1\}\cup\{i,j\}, \\
&\mbox{ with }n\leq i\leq j\mbox{ and }i\mbox{ is odd and }j\mbox{ is even}; \\
d&\mbox{if }T=\{1,\ldots,n-1\}\cup\{i,j\},\\
&\mbox{ with }n\leq j\leq i\mbox{ and }i\mbox{ is odd and }j\mbox{ is even}; \\
z&
\mbox{otherwise}.
\end{cases}
$$
By (i), the maps $\u_m$ are well defined. (Observe that if $n=1$, the first line is never applied. In this case (i)  simply states that $\Y$ is non-empty, which trivially holds.) Using the same argument as
in Lemma~\ref{Lem:PDLA} it is possible to prove that the sequence $\u_1,\u_2,\ldots$ is a sequence of unifiers for $\X$ satisfying condition (i) of Theorem~\ref{Theo:Baader}. Therefore, $\Type{}(\Un{\POS}(\X))=0$.
\end{proof}

\begin{lemma}\label{Lem:PDLC}
Let $\X$ be a finite poset. Assume there is $Y\in\max(\Con(\X))$,  $a,b,c,d\in Y$ and $n\in \N$ satisfying the following conditions
\begin{itemize}
\item[(i)] $\Y$ satisfies $(*_{n-1})$;
\item[(ii)] $|\min_{\Y}(a)\cup\min_{\Y}(b)|=n;$
\item[(iii)] $\min_{\Y}(a)\subsetneq\min_{\Y}(b)$;
\item[(iv)] $\min_{\Y}(c)=\min_{\Y}(d)=\min_{\Y}(a)\cup\min_{\Y}(b)$;
\item[(v)] $a,b\leq c,d$;
\item[(vi)] there is no $e\in Y$ such that $a,b\leq e\leq c,d$.
\end{itemize}
Then $\Type{}(\Un{\POS}(\X))=0$.
\end{lemma}
\begin{proof}
 As in the proof of Lemma~\ref{Lem:PDLB},  we present a sequence of unifiers $\u_1,\u_2,\ldots$ for $\X$ satisfying condition (i) of Theorem~\ref{Theo:Baader}, but we omit the details since they follow by similar arguments to the ones used in Lemma~\ref{Lem:PDLA}.

Let $x_1,\ldots,x_n$ be an enumeration of the elements of $\min_{\Y}(b)$. Without loss of generality, assume that $n$ is even and $x_{n-1}\in\min_{\Y}(x)$ and $x_{n}\in\min_{\Y}(b)\setminus\min_{\Y}(a)$.
Let $f\colon \N\to \{x_1,\ldots,x_n\}$ be defined by 
$$
f(i)=\begin{cases}
x_i& \mbox{if }i\leq n\\
x_{n-1}&\mbox{if } i> n\mbox{ and } i\mbox{ is odd},\\
x_n&\mbox{if } i> n\mbox{ and } i\mbox{ is even}.\\
 \end{cases}
$$
For each $m\in\N$ we define the maps $\u_{m}\colon\pos{P}(m)\to \Y$ as follows:
$$
\u_m(T)=\begin{cases}
\bigvee_{\Y}f(T)&\mbox{if }\min_{\Y}(a)\not\subseteq f(T)\mbox{ or }T\subsetneq\{1,\dots,n\}\\
a\vee_{\Y}\bigvee_{\Y}f(T)&\mbox{if }\min_{\Y}(a)\subseteq f(T)\mbox{, }f(T)\neq\{x_1,\ldots,x_n\}\\
&\mbox{ and }T\not\subseteq\{1,\dots,n\}\\
b&\mbox{if }T=\{1,\ldots,n-1\}\cup\{i\} \mbox{ for some even } i\geq n\\
c&\mbox{if }T=\{1,\ldots,n-2\}\cup\{i,j\},\\
&\mbox{ for some  }i\mbox{ odd and }j\mbox{ even, such that } n\leq i\leq j; \\
d&\mbox{if }T=\{1,\ldots,n-2\}\cup\{i,j\},\\
&\mbox{ for some  }i\mbox{ odd and }j\mbox{ even, such that } n\leq j\leq i; \\
z&
\mbox{otherwise}.
\end{cases}
$$
\end{proof}

\begin{lemma}\label{Lem:PDL}
Let $\X$ be a non-empty finite poset. Assume there exist $Y\subseteq \X$   satisfying the following conditions
\begin{itemize}
\item[(i)] $Y$ is a maximal element of $\Con(\X)$; 
\item[(ii)] $\Y$ does not satisfy $\Cstar$.
\end{itemize}
Then $\Type{}(\Un{\POS}(\X))=0$.

\end{lemma}
\begin{proof}
Since $\X$ is non-empty $Y$ is non-empty. 
Let  $n\in\N$  be the minimal natural number such that 
$\Y$ does not satisfy $\Cnstar$. Since $Y$ satisfies $(*_0)$, by (ii) such $n$ must exist. 
Further, there are $a,b\in Y$ witnessing the failure of $\Cnstar$. 
Since $Y$ is connected, we have that $a,b$ are such that $|\min_{\Y}(a)\cup\min_{\Y}(b)|=n$,  and they do not have a lowest upper bound, equivalently,  there exist $c,d\geq a,b$ such that $$\textstyle\min_{\Y}(c)=\min_{\Y}(d)=\min_{\Y}(a)\cup\min_{\Y}(b),$$ but there is not $e\in Y$ such that $a,b\leq e\leq c,d$.
Observe that if there is $f\in\X$ such that  $a,b\leq f\leq c,d$,  then $Y\cup\{f\}\in\Con(\X)$ which contradicts the maximality of $Y$. 

Now the proof divides in four cases: 
\begin{itemize}
\item[(a)] $\min_{\Y}(a)=\min_{\Y}(b)$;
\item[(b)] $\min_{\Y}(a)\subsetneq\min_{\Y}(b)$; 
\item[(c)] $\min_{\Y}(b)\subsetneq\min_{\Y}(a)$; and 
\item[(d)] $\min_{\Y}(a)\not\subseteq\min_{\Y}(b)$ and $\min_{\Y}(b)\not\subseteq\min_{\Y}(a)$.
\end{itemize}

If (a) holds, then Lemma~\ref{Lem:PDLB} implies that $\Type{}(\Un{\POS}(\X))=0$. Similarly, if (b) or (c) holds, Lemma~\ref{Lem:PDLC} gives the same conclusion. Finally, Lemma~\ref{Lem:PDLA} proves that $\Type{}(\Un{\POS}(\X))=0$, if (d) is the case.
\end{proof}

We are now ready to prove the main result of this section.
\begin{theorem}\label{Theo:ClassPDL}
Let $\X$ be a non-empty poset in $\POS$. Then 
$$
\Type{}(\Un{\POS}(\X))=\begin{cases}
|\max(\Con(\X))|& \mbox{ if each }\pos{Y}\in \max(\Con(\X)) \mbox{ satisfies } \Cstar;\\
0 & \mbox{otherwise}.
\end{cases}
$$
\end{theorem}
\begin{proof}
Observe that since $\X$ is finite and non-empty, also $\max(\Con(\X))$ is finite and non-empty.

Assume first that $\Y$ satisfies $\Cstar$ for each $Y\in \max(\Con(\X))$.  
Then, for each ${Y\in \max(\Con(\X))}$, the inclusion map $\mu_{\Y}\colon \Y\to \X$ is in $\Un{\POS}(\X)$. 
If $\u\colon\Z\to\X\in\Un{\POS}(\X) $, by Lemma 
~\ref{Lem:RangCon}, then $\u(\Z)\in\Con(\X)$. Hence there exists $Y\in \max(\Con(\X))$ such that $\u(\Z)\subseteq Y$. 
This implies that $\u\leqn \mu_{\Y}$. 
Moreover, if  $Y_1$ and $Y_2$ belong  to $\max(\Con(\X))$ and $Y_1\neq Y_2$, then $\mu_{\Y_1}\not\leqn \mu_{\Y_2}$. If we assume the contrary, there is a $p$-morphism $\psi\colon \Y_1\to \Y_2$ such that  $\mu_{\Y_1}= \mu_{\Y_2}\circ \psi$. Then \[Y_1=\mu_{\Y_1}(Y_1)= \mu_{\Y_2}\circ \psi(Y_1)\subseteq \mu_{\Y_2}(Y_2)=Y_2.\]
 Since $Y_1$ is maximal in $\Con(\X)$, we have  $Y_1=Y_2$, a contradiction. 
This proves that the set $\{\mu_{\Y}\mid Y\in \max(\Con(\X))\}$ is a minimal complete set in $\Un{\POS}(\X)$. It follows that
$\Type{}(\Un{\POS}(\X))=|\max(\Con(\X))|$.

Now suppose that $Y\in \max(\Con(\X))$ is such that $\Y$ does not satisfy $\Cstar$. 
By Lemma~\ref{Lem:PDL}, $\Type{}(\Un{\POS}(\X))=0$.
\end{proof}

\section{Type of unification problems in 
$\mathfrak{B}_{n}$ 
with $2\leq n<\omega$}\label{Sec:UnifPDLn}
In this final section we compute the type $\Un{\POn}(\X)$ when $\X\in \POn$ with $n\geq2$.
We first obtain a family of conditions for a poset in $\POn$ to have unification type~$0$ (Lemmas~\ref{Lem:CaseAI}-\ref{Lem:CaseBIII}). Collectively these conditions imply that if a poset $\X\in\POn$  has an $n$-connected subset $Y$ that is maximal in $\Cn(\X)$ and such that $\Y$ does not satisfy~$\Cnstar$, then the type of $\Un{\POn}(\X)$ is $0$ (Lemma~\ref{Lem:Type0PDLn}). Finally in Theorem~\ref{Theo:ClasPDLn} we present our description of the unification type of all posets in $\POn$.

\begin{lemma}\label{Lem:CaseAI}
Let $\X$ be a poset in $\POn$. Assume there exist $Y\in\max(\Cn(\X))$,  $a,b\in Y$ and $k\in \N$ satisfying the following
\begin{itemize}
\item[(i)] $\Y$ satisfies condition $(*_{k-1})$;
\item[(ii)]$|\min_{\Y}(a)\cup\min_{\Y}(b)|=k\leq n$;
\item[(iii)] $\min_{\Y}(a)=\min_{\Y}(b)$;
\item[(iv)] if $c\in Y$ is such that $c\geq a, b$, then $\min_{\Y}(c)\neq\min_{\Y}(a)\cup\min_{\Y}(b)$.
\end{itemize}
Then $\Type{}(\Un{\POn}(\X))=0$.
\end{lemma}
\begin{proof}
Let $x_1,\ldots,x_k$ be an enumeration of $\min_{\Y}(a)\cup\min_{\Y}(b)$.

Since $Y$ is $n$-connected it is enough to consider the case when there exists $z\in Y$ such that $a\geq z\leq b$ and $\min_{\Y}(z)=\min_{\Y}(a)=\min_{\Y}(b)$. 

Let  $f\colon \N\to \{x_1,\ldots,x_k\}$ be defined by 
$$
f(i)=\begin{cases}
x_i& \mbox{if }i\leq k;\\
x_k&\mbox{if } i> k.
 \end{cases}
$$

For each $m\in\N$ we define the maps $\u_{m}\colon(\pos{P}(m))_n\to \X$ as follows:
$$
\u_{m}(T)=\begin{cases}
a&\mbox{if }f(T)=\{x_1,\ldots,x_k\}, |T|=n;\\
&\mbox{ and } \min(T\setminus\{1,\ldots k-1\})\mbox{ is odd};\\
b&\mbox{if }f(T)=\{x_1,\ldots,x_k\}, |T|=n\\
&\mbox{ and } \min(T\setminus\{1,\ldots k-1\})\mbox{ is even}\\
z&\mbox{if }f(T)=\{x_1,\ldots,x_k\}\mbox{ and } |T|<n;\\
\bigvee_{\Y}f(T)&
\mbox{otherwise}.
\end{cases}
$$
Observe that, on the one hand if $k=1$, then the last line of the definition of $\u_m$ is never applied. On the other hand, if $k>1$, by (i), $\Y$ satisfies $(*_{k-1})$, hence $\bigvee_{\Y}f(T)$ exists for each $T$ such that $|f(T)|\leq k-1$. Thus the map $\u_m$ is well defined and  satisfies $\min_{\X}(\u_m(T))=\min_{\Y}(\u_m(T))=f(T)$. It is easy to see that each $\u_m$ is order preserving, and a $p$-morphism. Therefore each $\u_m$ is a well-defined element of $\Un{\POn}(\X)$, and  $\u_1\leqn\u_2\leqn \cdots$.

Suppose we are given $\u\colon\Z\to\X$ in $\Un{\POn}(\X)$ and $\psi\colon(\Pm)_n\to \Y$ in such a way that $\u_m=\u\circ\psi$ for some $m> 3n$.
 We claim:
\begin{itemize}
\item[(a)]  $|Z|\geq m-n$; and 
\item[(b)] there exists $\nu\in\Un{\POn}(\X)$ such that $\u_{m+1},\u\leqn \nu $.
\end{itemize}

First assume that $|Z|< m-n$. Then there exist $i,j\in\{n+1,\ldots,m\}$ such that $\psi(\{i\})=\psi(\{j\})$ and $i< j$. 
Let us consider the sets
$T_1=(\{1,\ldots,n\}\cup\{i\})\setminus\{k\}$ and $T_2=(\{1,\ldots,n\}\cup\{j\})\setminus\{k+1\}$.
If  $k$ is even, $\u_m(T_1)=a$ and $\u_m(T_2)=b$. If $k$ is odd, then $\u_m(T_1)=b$ and $\u_m(T_2)=a$. In both cases, $|\min_{\Z}\psi(T_1)\cup\min_{Z}(T_2)|\leq n$.
Since $\Z$ satisfies $\Cnstar$, there is $x\in \Z$ such that $x\geq \psi(T_1),\psi(T_2)$ and 
\[\textstyle\min_{\Z}(x)= \min_{\Z}(\psi(T_1))\cup\min_{\Z}(\psi(T_2)).\]
Then $\u(x)\geq \u(\psi(T_1)),\u(\psi(T_2))$, that is $\u(x)\geq a,b$ and $\min_{\Y}(\u(x))=\{x_1,\ldots,x_k\}$. By (iv), $\u(x)\notin Y$ and by Lemma~\ref{Lem:UpSetNCon}, $Y\cup\{\u(x)\}\in\Cn(\X)$. This contradicts the maximality of $Y$ in $\Cn(\pos{X})$. From this claim, it follows that the sequence $\u_1,\u_2,\ldots$ of unifiers of $\X$ does not admit an upper bound.

To prove claim (b), assume $m$ is odd (the case $m$ even follows by a simple modification of this argument).
In the proof of claim (a) we observed that $\psi(\{i\})\neq\psi(\{j\})$ for each $i,j\in\{n+1,\ldots,m\}$. Hence, since $m>3n$, for each $u\in\Z$ such that $|\min_{\Z}(u)|\leq n-1$, the set $\{j\in\{k,\ldots,m\}\mid j\mbox{ is even and } \psi(\{j\})\not\leq u\}$ is non-empty.
Letting $j_u=\max\{j\in\{k,\ldots,m\}\mid j\mbox{ is even and } \psi(\{j\})\not\leq u\}$, 
we define 
\[\textstyle u'=u\bigvee_{\Z}\{\psi(S\cup\{j_u\})\mid S=\emptyset\mbox{ or }(S\in(\pos{P}(s))_{n-1} \mbox{ and  }\psi(S)\leq u)\}.
\]
Since $|\min_{\Z}(u)\cup\min_{\Z}\psi(\{j_u\})|\leq n $ and  $\Z$ satisfies $\Cnstar$, the existence of $u'$ is granted.

Let $\nu\colon (\Z\odot\POne)_n\to\Y$ defined by:
$$
\nu(u,S)=\begin{cases}
\u(u)&\mbox{if }S=\bot;\\
x_{k}&\mbox{if }u=\bot;\\
\u(u')&\mbox{ otherwise}.
\end{cases}
$$
The equality $\nu(\min_{\Z\odot\POne}(u,S))=\min_{\Y}(\nu(u,S))$ follows from the definition of $u'$ and the fact that $\u$ is a $p$-morphism.
Let $(u,S),(w,T)\in(\Z\odot\POne)_n$ be such that $(u,S)\leq (w,T)$. If $S=\bot$, then $u,w\in\Y$, and  $\nu(u,S)=\u(u)\leq \u(w)\leq \nu(w,T)$. If $a=\{1\}$ and $u\neq\bot$ then $\nu(u,S)=\u(u')\leq \u(w')=\nu(w,T)$. If $u=\bot$ and $w\in \Y$, then $\nu(u,S)=x_k=\u(\psi(\{j_w\}))\leq \u(w\vee_{\Z}\psi(\{j_w\}))$. Since $m$ is odd, $m-1$ is even and $w\vee_{\Z}\psi(\{m-1\})\leq w'$. 
Therefore  $\nu$ is order preserving and a $p$-morphism.

It is straightforward to check that $\nu\circ\iota_{\Y}=\u$. Which implies that $\u\leqn \nu$. 

We  now prove $\u_{m+1}\leqn \nu$. Indeed we claim  $\nu\circ(\psi\odot{\rm Id}_{\POne})\circ\eta_{m,1}=\u_{m+1}$.
Let $T\in (\pos{P}(m+1))_n$. If $m+1\notin T$. Then 
\begin{align*}
(\nu\circ(\psi\odot{\rm Id}_{\POne})\circ\eta_{m,1})(T)&=
(\nu\circ(\psi\odot{\rm Id}_{\POne}))(T,\bot)\\
&=\nu(\psi(T),\bot)=\u_{m}(T)=\u_{m+1}(T).
\end{align*} 
If $T=\{m+1\}$, $$(\nu\circ(\psi\odot{\rm Id}_{\POne})\circ\eta_{m,1})(T)=\nu(\bot,\{1\})=x_k=f(m+1)=\nu_{m+1}(\{m+1\}).$$
If $m+1\in T$ and $T\neq\{m+1\}$, from the fact that $|T|\leq n$ and $m>3n$, it follows that the set $S=\{j\in\{n,\ldots,m\}\mid j\mbox{ even and } j\notin T\}\neq\emptyset$. Let $i=\max(S)$ and $T'=\{i\}\cup T\setminus\{m+1\}$. Thus
 \begin{align*}
(\nu\circ(\psi\odot{\rm Id}_{\POne})\circ\eta_{m,1})(T)&=
(\nu\circ(\psi\odot{\rm Id}_{\POne}))(T\setminus\{m+1\},\{1\})\\
&=\nu(\psi(T\setminus\{m+1\}),\{1\})\\
&=\u_{m}(T')=\u_{m+1}(T).
\end{align*} 

Combining (a), (b)  and Theorem~\ref{Theo:Baader}, it follows that ${\Type{}(\Un{\POS}(\X))=0}$.
\end{proof}

The statements and proofs of Lemmas~\ref{Lem:CaseAII}-\ref{Lem:CaseBIII} are similar to the statement and proof of Lemma~\ref{Lem:CaseAI}.  In each of the proofs of these lemmas, the most delicate part is to find a sequence of unifiers that satisfies  condition (i) in Theorem~\ref{Theo:Baader}. Proving that each sequence actually satisfies that condition is achieved with a similar  argument to the one used in Lemma~\ref{Lem:CaseAI}. Therefore, we shall only present these sequences of unifiers in each case and omit the details.

\begin{lemma}\label{Lem:CaseAII}
Let $\X$ be a poset in $\POn$. Assume there exist $Y\in\max(\Cn(\X))$,  $a,b\in Y$ and $k\in \N$ satisfying the following
\begin{itemize}
\item[(i)]  $\Y$ satisfies condition $(*_{k-1})$;
\item[(ii)]$|\min_{\Y}(a)\cup\min_{\Y}(b)|=k\leq n$;
\item[(iii)] $\min_{\Y}(a)\subsetneq \min_{\Y}(b)$;
\item[(iv)] if $c\in Y$ is such that $c\geq a, b$, then $\min_{\Y}(c)\neq\min_{\Y}(a)\cup\min_{\Y}(b)$.
\end{itemize}
Then $\Type{}(\Un{\POn}(\X))=0$.
\end{lemma}
\begin{proof}
Let $x_1,\ldots,x_k$ be an enumeration of $\min_{\Y}(a)\cup\min_{\Y}(b)$.
Without loss of generality assume $\min_{\Y}(a)\subsetneq\min_{\Y}(b)$ and  $x_k\in \min_{\Y}(a)$.

Let $f\colon \N\to \{x_1,\ldots,x_k\}$ be defined by 
$$
f(i)=\begin{cases}
x_i& \mbox{if }i\leq k;\\
x_k&\mbox{if } i> k.
 \end{cases}
$$
For each $m\in\N$ we define the maps $\u_{m}\colon(\pos{P}(m))_n\to \X$ as follows:
$$
\u_m(T)=\begin{cases}
a&\mbox{if }f(T)=\min_{\Y}(x)\mbox{ and }|T|=n;\\
b&\mbox{if }f(T)=\min_{\Y}(y);\\
\bigvee_{\Y}(f(T))&
\mbox{otherwise}.
\end{cases}
$$
A similar argument to the one used in the proof of Lemma~\ref{Lem:CaseAI} proves that the sequence $\u_1,\u_2,\ldots$ satisfies condition (i) of Theorem~\ref{Theo:Baader}. From this we conclude that  $\Type{}(\Un{\POn}(\X))=0$.
\end{proof}

\begin{lemma}\label{Lem:CaseAIII}
Let $\X$ be a poset in $\POn$. Assume there exist $Y\in\max(\Cn(\X))$,  $a,b\in Y$ and $k\in \N$ satisfying the following
\begin{itemize}
\item[(i)]  $\Y$ satisfies condition $(*_{k-1})$;
\item[(ii)]$|\min_{\Y}(a)\cup\min_{\Y}(b)|=k\leq n$;
\item[(iii)] $\min_{\Y}(a)\not\subseteq \min_{\Y}(b)$ and $\min_{\Y}(b)\not\subseteq \min_{\Y}(a)$;
\item[(iv)] if $c\in Y$ is such that $c\geq a, b$, then $\min_{\Y}(c)\neq\min_{\Y}(a)\cup\min_{\Y}(b)$.
\end{itemize}
Then $\Type{}(\Un{\POn}(\X))=0$.
\end{lemma}
\begin{proof}
Let $x_1,\ldots,x_k$ be an enumeration of $\min_{\Y}(a)\cup\min_{\Y}(b)$.
Without loss of generality, assume that $x_{k-1}\in\min_{\Y}(a)\setminus\min_{\Y}(b)$ and 
$x_k\in\min_{\Y}(b)\setminus\min_{\Y}(a)$. Since $Y$ is $n$-connected there exists $z\in Y$ such that $\min_{\Y}(z)=\{x_1,\ldots,x_k\}$.

Let $f\colon \N\to \{x_1,\ldots,x_k\}$ be defined by 
\[
f(i)=\begin{cases}
x_i& \mbox{if }i\leq k-2\\
x_{k-1}&\mbox{if } i> k-2\mbox{ and } i\mbox{ is odd},\\
x_k&\mbox{if } i> k-2\mbox{ and } i\mbox{ is even}.\\
 \end{cases}
\]

For each $m\in\N$ we define the maps $\u_{m}\colon(\pos{P}(m))_n\to \X$ as follows:
$$
\u_m(T)=\begin{cases}
a&\mbox{if }f(T)=\min_{\Y}(x)\mbox { and } |T|=n;\\
b&\mbox{if }f(T)=\min_{\Y}(y)\mbox { and } |T|=n;\\
z&\mbox{if }f(T)=\min_{\Y}(y)=\{x_1,\ldots,x_k\};\\
\bigvee_{\Y}(f(T))&
\mbox{otherwise}.
\end{cases}
$$
The proof now follows similarly to the proof of Lemma~\ref{Lem:CaseAI}.
\end{proof}

\begin{lemma}\label{Lem:CaseBI}
Let $\X$ be a poset in $\POn$. Assume there exist $Y\in\max(\Cn(\X))$,  ${a,b,c,d\in Y}$ and $k\in \N$ satisfying the following
\begin{itemize}
\item[(i)] $\Y$ satisfies condition $(*_{k-1})$;
\item[(ii)]$|\min_{\Y}(a)\cup\min_{\Y}(b)|=k\leq n$;
\item[(iii)] $\min_{\Y}(a)= \min_{\Y}(b)$;
\item[(iv)] $\min_{\Y}(c)=\min_{\Y}(d)=\min_{\Y}(a)\cup\min_{\Y}(b)$;
\item[(v)] $a,b\leq c,d$;
\item[(vi)] there is no $e\in Y$ such that $a,b\leq e\leq c,d$.
\end{itemize}
Then $\Type{}(\Un{\POn}(\X))=0$.
\end{lemma}
\begin{proof}
Let $x_1,\ldots,x_k$ be an enumeration of the elements of $\min_{\Y}(a)=\min_{\Y}(b)$. If there is no $z\in Y$ such that $c,d\leq z$, then the result follows by an application of Lemma~\ref{Lem:CaseAI}. Therefore, we can assume there exists $z\in Y$ such that $c,d\leq z$.

Let $f\colon \N\to \{x_1,\ldots,x_k\}$ be defined by 
$$
f(i)=\begin{cases}
x_i& \mbox{if }i\leq k-1,\\
x_{k-1}&\mbox{if }  k-1<i<n,\\
x_{k}&\mbox{if }  n\leq i,\\
 \end{cases}
$$

For this particular case, a slightly different class of posets satisfying $\Cnstar$ is needed. Let $\pos{2}=(\{0,1\},\leq)$ denote the poset such that $0<1$. For each $m\in\N$, let $\pos{Q}(m)$ denote the poset $(\pos{P}(m))_n\times\pos{2}\times \pos{2}$. If   $(U,v,w)\in\pos{Q}(m)$, then $|U|\leq n$ and
$\min_{\pos{Q}(m)}(U,v,w)=\{(\{n\},0,0)\mid n\in U\}
$. Therefore $\pos{Q}(m)\in\POn$ and it satisfies $\Cnstar$.
We now define the sequence of unifiers $\u_{m}\colon\pos{Q}(m)\to \X$ as follows:
$$
\u_m(T,v,w)=\begin{cases}
\bigvee_{\Y}f(T)&\mbox{if }f(T)\neq\{x_1,\ldots,x_k\};\\
a&\mbox{if } T=\{1,\ldots,n-1\}\cup\{i\}, i\geq n \mbox{  is odd}\\
&\mbox{ and } (v,w)=(0,0)\\
b&\mbox{if } T=\{1,\ldots,n-1\}\cup\{i\}, i\geq n \mbox{  is even,}\\
&\mbox{ and } (v,w)=(0,0)\\
c&\mbox{if }T=\{1,\ldots,n-1\}\cup\{i\}, i\geq n  \\
&\mbox{ and } (v,w)=(1,0)\\
d&\mbox{if }T=\{1,\ldots,n-1\}\cup\{i\}, i\geq n  \\
&\mbox{ and } (v,w)=(0,1)\\
z&
\mbox{otherwise}.
\end{cases}
$$
\end{proof}

\begin{lemma}\label{Lem:CaseBII}
Let $\X$ be a poset in $\POn$. Assume there exist $Y\in\max(\Cn(\X))$,  ${a,b,c,d\in Y}$ and $k\in \N$ satisfying the following
\begin{itemize}
\item[(i)] $\Y$ satisfies condition $(*_{k-1})$;
\item[(ii)]$|\min_{\Y}(a)\cup\min_{\Y}(b)|=k\leq n$;
\item[(iii)] $\min_{\Y}(a)\subsetneq \min_{\Y}(b)$;
\item[(iv)] $\min_{\Y}(c)=\min_{\Y}(d)=\min_{\Y}(a)\cup\min_{\Y}(b)$;
\item[(v)] $a,b\leq c,d$;
\item[(vi)] there is no $e\in Y$ such that $a,b\leq e\leq c,d$.
\end{itemize}
Then $\Type{}(\Un{\POn}(\X))=0$.
\end{lemma}
\begin{proof}
As in Lemma~\ref{Lem:CaseBII}, we can assume there exists $z\in Y$ such that $c,d\leq z$.

Let $x_1,\ldots,x_k$ be an enumeration of the elements of $\min_{\Y}(b)$. 
Without loss of generality assume that  $x_{n-1}\in\min_{\Y}(x)$ and $x_{n}\in\min_{\Y}(y)\setminus\min_{\Y}(x)$.
Let $f\colon \N\to \{x_1,\ldots,x_n\}$ be defined by 
$$
f(i)=\begin{cases}
x_i& \mbox{if }i\leq k-2\\
x_{k-2}& \mbox{if }k-2<i\leq n-2\\
x_{n-1}&\mbox{if } i> n\mbox{ and } i\mbox{ is odd},\\
x_n&\mbox{if } i> n\mbox{ and } i\mbox{ is even}.\\
 \end{cases}
$$
For each $m\in\N$ we define the unifiers $\u_{m}\colon(\pos{P}(m))_n\to \Y$ as follows:
$$
\u_m(T)=\begin{cases}
\bigvee_{\Y}f(T)&\mbox{if }\min_{\Y}(a)\not\subseteq f(T)\mbox{ or }T\subsetneq\{1,\dots,n\}\\
a\vee_{\Y}\bigvee_{\Y}f(T)&\mbox{if }\min_{\Y}(a)\subseteq f(T)\mbox{ and }f(T)\neq\{x_1,\ldots,x_n\}\\
&\mbox{ and }T\not\subseteq\{1,\dots,n\}\\
b&\mbox{if }T=\{1,\ldots,n-1\}\cup\{i\} \mbox{ for some even } i\geq n\\
c&\mbox{if }T=\{1,\ldots,n-2\}\cup\{i,j\},\\
&\mbox{ for some  }i\mbox{ odd and }j\mbox{ even, such that } n\leq i\leq j; \\
d&\mbox{if }T=\{1,\ldots,n-2\}\cup\{i,j\},\\
&\mbox{ for some  }i\mbox{ odd and }j\mbox{ even, such that } n\leq j\leq i; \\
z&
\mbox{otherwise}.
\end{cases}
$$
\end{proof}

\begin{lemma}\label{Lem:CaseBIII}
Let $\X$ be a poset in $\POn$. Assume there exist $Y\in\max(\Cn(\X))$,  ${a,b,c,d\in Y}$ and $k\in \N$ satisfying the following
\begin{itemize}
\item[(i)] $\Y$ satisfies condition $(*_{k-1})$;
\item[(ii)]$|\min_{\Y}(a)\cup\min_{\Y}(b)|=k\leq n$;
\item[(iii)] $\min_{\Y}(a)\not\subseteq \min_{\Y}(b)$ and $\min_{\Y}(b)\not\subseteq \min_{\Y}(a)$;
\item[(iv)] $\min_{\Y}(c)=\min_{\Y}(d)=\min_{\Y}(a)\cup\min_{\Y}(b)$;
\item[(v)] $a,b\leq c,d$;
\item[(vi)] there is no $e\in Y$ such that $a,b\leq e\leq c,d$.
\end{itemize}
Then $\Type{}(\Un{\POn}(\X))=0$.
\end{lemma}
\begin{proof}
Let $x_1,\ldots,x_k$ be an enumeration of the elements of $\min_{\Y}(a)\cup\min_{\Y}(b)$ such that $x_{k-1}\in\min_{\Y}(a)\setminus\min_{\Y}(b)$ and $x_{k}\in\min_{\Y}(b)\setminus\min_{\Y}(a)$.
Let $f\colon \N\to \{x_1,\ldots,x_n\}$ be defined by 
$$
f(i)=\begin{cases}
x_i& \mbox{if }i\leq k-2\\
x_{k-1}&\mbox{if } i> k-2\mbox{ and } i\mbox{ is odd},\\
x_k&\mbox{if } i> k-2\mbox{ and } i\mbox{ is even}.\\
 \end{cases}
$$

For each $m\in\N$, we define $\u_{m}\colon(\Pm)_n\to \X$ as follows:
$$
\u_m(T)=\begin{cases}
\bigvee_{\Y}(f(T))&\mbox{if }\min_{\Y}(a)\not\subseteq f(T)\mbox{ and }\min_{\Y}(b)\not\subseteq f(T);\\
a\vee_{\Y}\bigvee_{\Y}(f(T))&\mbox{if }\min_{\Y}(a)\subseteq f(T)\neq\{x_1,\ldots,x_k\};\\
b\vee_{\Y}\bigvee_{\Y}(f(T))&\mbox{if }\min_{\Y}(b)\subseteq f(T)\neq\{x_1,\ldots,x_k\};\\
c&\mbox{if }T=\{1,\ldots,k-2\}\cup\{i,j\},\\
&\mbox{ with }k-2<i\leq j\mbox{ and }i\mbox{ is odd and }j\mbox{ is even}; \\
d&\mbox{if }T=\{1,\ldots,k-2\}\cup\{i,j\},\\
&\mbox{ with }k-2<j\leq i\mbox{ and }i\mbox{ is odd and }j\mbox{ is even}; \\
z&
\mbox{otherwise}.
\end{cases}
$$
\end{proof}

The results of Lemmas~\ref{Lem:CaseAI}-\ref{Lem:CaseBIII} are the core of the proof of the following lemma.
\begin{lemma}\label{Lem:Type0PDLn}
Let $\X$ be a poset in $\POn$. Assume there exists $Y\subseteq \X$   satisfying the following conditions
\begin{itemize}
\item[(i)] $Y$ is a maximal element of $\Cn(\X)$; 
\item[(ii)] $\Y$ does not satisfy condition $\Cnstar$.
\end{itemize}
Then $\Type{}(\Un{\POn}(\X))=0$.
\end{lemma}
\begin{proof}
Let $k\in\N$ be the minimal natural number such that  $\Y$ does not satisfy condition $(*_k)$. By (ii), $k\leq n$. 
Let $a,b\in Y$ witnessing the failure of $(*_{k})$. 
More precisely, $a,b$ satisfy $|\min_{\Y}(a)\cup\min_{\Y}(b)|=k$;  and one of the following conditions hold
\begin{itemize}
\item[(A)] $\min_{\Y}(z)\neq\min_{\Y}(a)\cup\min_{\Y}(b)$, for each $z\geq a, b$ ; or
\item[(B)] there are $c,d\in\Y\subseteq\X$ such that 
$$\textstyle\min_{\Y}(c)=\min_{\Y}(d)=\min_{\Y}(a)\cup\min_{\Y}(b)$$ but there is no $e\in \Y$ such that $a,b\leq e\leq c,d$.
\end{itemize}

Each of these cases splits in four sub-cases depending on the relation between $\min_{\Y}(a)$ and $\min_{\Y}(b)$
\begin{itemize}
\item[(I)] $\min_{\Y}(a)=\min_{\Y}(b)$;
\item[(II)] $\min_{\Y}(a)\subsetneq\min_{\Y}(b)$
\item[(III)] $\min_{\Y}(b)\subsetneq\min_{\Y}(a)$; 
\item[(IV)] $\min_{\Y}(a)\not\subseteq\min_{\Y}(b)$ and $\min_{\Y}(b)\not\subseteq\min_{\Y}(a)$.
\end{itemize}

For case (A), Lemmas~\ref{Lem:CaseAI} and \ref{Lem:CaseAIII}, prove that $\Type{}(\Un{\POn}\X)=0$ for the subcases (I) and (III), respectively. The same conclusion follows for Cases (A)(II-III) from Lemma~\ref{Lem:CaseAII}. 
If (B) holds then we obtain ${\Type{}(\Un{\POn}\X)=0}$ for the subcases (I), (II-III), and (III), from Lemmas~\ref{Lem:CaseBI}, \ref{Lem:CaseBII}, and \ref{Lem:CaseBIII}, respectively. 
\end{proof}

We are now ready to present the main result of the section.
\begin{theorem}\label{Theo:ClasPDLn}
Let $\X$ be a non-empty poset in $\POn$. Then 
$$
\Type{}(\Un{\POn}\X)=\begin{cases}
|\max(\Cn(\X))|& \mbox{ if each } \Y\in \max(\Cn(\X))\mbox{ satisfies } \Cnstar;\\
0 & \mbox{otherwise}.
\end{cases}
$$
\end{theorem}
\begin{proof}
Since $\X$ is finite and non-empty, also $\max(\Cn(\X))$ is finite and non-empty. 

Assume first that $\Y$ satisfies $\Cnstar$ for each $Y\in \max(\Cn(\X))$. 
The maps $\mu_{\Y}\colon \Y\to \X$ are in $\Un{\POn}(\X)$. 
Now if $\u\colon\Z\to\X\in\Un{\POn}(\X) $, by Lemma~\ref{Lem:RangN-Con}, we have $\u(\Z)\in \Cn(\X)$. 
Therefore, there is $Y\in  \max(\Cn(\X))$ such that $\u(\Z)\subseteq Y$. 
This implies that $\u\leqn \mu_{\Y}$. 
It is easy to see that whenever $Y,Z\in\max(\Cn(\X))$ are different then $\mu_{\Y}\not\leqn \mu_{\Z}$ and $\mu_{\Z}\not\leqn \mu_{\Y}$. Hence, the set $\{\mu_{\Y}\mid Y\in\max(\Cn(\X))\}$ is a minimal complete set in $\Un{\POS}(\X)$, and 
$ \Type{}(\Un{\POn}(\X))=|\max(\Cn(\X))|$.

Now assume there exists $\Y\in \max(\Cn(\X))$ that does not satisfy $\Cnstar$. In this case, Lemma~\ref{Lem:Type0PDLn} proves that $\Type{}(\Un{\POn}(\X))=0$. 
\end{proof}

\subsection*{Acknowledgments}
We would like to thank Hilary Priestley  and Daniele Mundici for their valuable comments and suggestions on a previous draft of this paper. 
We are deeply indebted with both referees for their careful reading  of this paper. Their useful remarks and suggestions have lead to a complete reorganisation of the structure of the paper, which greatly simplified the presentation of our results and their proofs.


\end{document}